\documentclass[a4paper,10pt]{amsart}
\usepackage[top=0.8in, bottom=0.8in, left=1.2in, right=1.2in]{geometry}
\usepackage{fancybox}
\usepackage{amsmath}
\usepackage{amssymb}
\usepackage{amsthm}
\usepackage{mathrsfs}
\usepackage{epstopdf}
\usepackage[colorinlistoftodos]{todonotes}
\usepackage{exscale}
\usepackage{relsize}
\usepackage[all]{xy}
\usepackage{epstopdf}
\usepackage{array}
\usepackage{amscd}

\usepackage{hyperref}

\usepackage{amsfonts}
\usepackage{amssymb}
\usepackage{mathrsfs}
\usepackage{tikz}

\DeclareFontFamily{U}{rsfs}{%
\skewchar\font127}
\DeclareFontShape{U}{rsfs}{m}{n}{%
<-6>rsfs5<6-8.5>rsfs7<8.5->rsfs10}{}
\DeclareSymbolFont{rsfs}{U}{rsfs}{m}{n}
\DeclareSymbolFontAlphabet
{\mathrsfs}{rsfs}
\DeclareRobustCommand*\rsfs{%
\@fontswitch\relax\mathrsfs}

\theoremstyle{plain}
\newtheorem{thm}{Theorem}[section]
\newtheorem{prop}[thm]{Proposition}
\newtheorem{lem}[thm]{Lemma}

\newtheorem{prop-defi}[thm]{Proposition-Definition}
\newtheorem{thm-defi}[thm]{Theorem-Definition}
\newtheorem{lem-defi}[thm]{Lemma-Definition}
\newtheorem{cor-defi}[thm]{Corollary-Definition}

\newtheorem{conj}[thm]{Conjecture}

\theoremstyle{definition}
\newtheorem{defi}[thm]{Definition}
\newtheorem{rmk}[thm]{Remark}
\newtheorem{exam}[thm]{Example}

\newdimen\argwidth
\def\db[#1\db]{
 \setbox0=\hbox{$#1$}\argwidth=\wd0
 \setbox0=\hbox{$\left[\box0\right]$}
  \advance\argwidth by -\wd0
 \left[\kern.3\argwidth\box0 \kern.3\argwidth\right]}

\newcommand{\lL}{\mathcal{L}}

\newcommand{\oO}{\mathcal{O}}

\newcommand{\C}{\mathbb{C}}
\newcommand{\Q}{\mathbb{Q}}
\newcommand{\Z}{\mathbb{Z}}
\newcommand{\vir}{\mathrm{vir}}

\newcommand{\Hom}{\mathop{\rm Hom}\nolimits}

\newcommand{\dR}{\mathbf{R}}

\newcommand{\Hilb}{\mathop{\rm Hilb}\nolimits}

\newcommand{\rk}{\mathop{\rm rk}\nolimits}

\newcommand{\Ext}{\mathop{\rm Ext}\nolimits}
\newcommand{\Spec}{\mathop{\rm Spec}\nolimits}

\newcommand{\cneq}{\mathrel{\raise.095ex\hbox{:}\mkern-4.2mu=}}
\newcommand{\eqcn}{\mathrel{=\mkern-4.5mu\raise.095ex\hbox{:}}}

\newcommand{\ext}{\mathop{\rm ext}\nolimits}

\newcommand{\DT}{\mathop{\rm DT}\nolimits}

\newcommand{\tr}{\mathop{\rm tr}\nolimits}

\newcommand{\RHom}{\mathop{\dR\mathrm{Hom}}\nolimits}

\makeatletter
 
  \@addtoreset{equation}{section}
\makeatother

\begin{document}
\title[Counting zero-dimensional subschemes in higher dimensions]
{Counting zero-dimensional subschemes \\ in higher dimensions}
\author{Yalong Cao}
\address{Mathematical Institute, University of Oxford, Andrew Wiles Building, Radcliffe Observatory Quarter, Woodstock Road, Oxford, OX2 6GG }
\email{yalong.cao@maths.ox.ac.uk}
\author{Martijn Kool}
\address{Mathematical Institute, Utrecht University, P.O. Box 80010 3508 TA Utrecht, The Netherlands}
\email{m.kool1@uu.nl}

\maketitle
\begin{abstract}
Consider zero-dimensional Donaldson-Thomas invariants of a toric threefold or toric Calabi-Yau fourfold. In the second case, invariants can be defined using a tautological insertion. In both cases, the generating series can be expressed in terms of the MacMahon function. In the first case, this follows from a theorem of Maulik-Nekrasov-Okounkov-Pandharipande. In the second case, this follows from a conjecture of the authors and a (more general $K$-theoretic) conjecture of Nekrasov.

In this paper, we consider formal analogues of these invariants in any dimension $d \not\equiv 2\ \mathrm{mod}\, 4$. The direct analogues of the above-mentioned conjectures \emph{fail} in general when $d>4$, showing that dimensions 3 and 4 are special. Surprisingly, after appropriate specialization of the equivariant parameters, the conjectures seem to hold in all dimensions. 
\end{abstract}

\section{Introduction}

Let $X$ be a smooth variety over $\C$ and denote by $\Hilb^n(X)$ the Hilbert scheme of $n$ points on $X$. When $\dim_{\C}(X)\leqslant 2$,  Hilbert schemes are smooth, in which case there is a vast literature on their geometry, topology, and representation theory.

\subsection*{Threefolds} When $\dim_{\C}(X) = 3$, the Hilbert scheme $\Hilb^n(X)$ is in general singular. Nevertheless, there exists a virtual class $[\Hilb^{n}(X)]^{\mathrm{vir}} \in H_0(\Hilb^n(X))$. When $X$ is projective, its degree is known as a Donaldson-Thomas invariant \cite{MNOP, Thomas}. 

When $X$ is toric with torus $(\C^*)^3$, we may define equivariant DT invariants \cite{MNOP} by the virtual localization formula: 
\begin{align*}
\mathrm{DT}_3(X,\,n):=\sum_{[Z] \in \Hilb^n(X)^{(\C^*)^3}} \frac{e_{(\C^*)^3}\big(\Ext^2(I_Z,I_Z)\big)}{e_{(\C^*)^3}\big(\Ext^1(I_Z,I_Z)\big)} \in \Q(\lambda_1,\lambda_2,\lambda_3),
\end{align*}
where the $(\C^*)^3$-fixed locus $\Hilb^{n}(X)^{(\C^*)^3}$ consists of finitely many reduced points, $I_Z$ is the ideal sheaf of $Z$, $e_{(\C^*)^3}(-)$ denotes the equivariant Euler class, and $\lambda_1,\lambda_2,\lambda_3$ are the equivariant parameters.

In \cite{MNOP}, Maulik-Okounkov-Nekrasov-Pandharipande proved the following formula:
$$
1+\sum_{n=1}^{\infty} \mathrm{DT}_3(X,\,n) \, q^n = M(-q)^{\int_X c^{(\C^*)^3}_3(TX \otimes K_X)},
$$
where $M(q) := \prod_{n>0} (1-q^n)^{-n}$ denotes the MacMahon function, $c^{(\C^*)^3}_3(-)$ is the equivariant third Chern class, $TX$ is the tangent bundle, and integration means equivariant push-forward to a point. 

\subsection*{Calabi-Yau fourfolds} When $X$ is a projective Calabi-Yau fourfold and assuming an orientability result \cite{CL2}, 
the Hilbert scheme $\Hilb^n(X)$ also carries a virtual class 
$
[\Hilb^n(X)]_{o(\lL)}^{\vir} \in H_{2n}(\Hilb^n(X))
$
(see \cite{BJ, CL}). This virtual class depends on a choice of orientation $o(\lL)$ on $\Hilb^n(X)$. Since the (real) virtual dimension is $2n$, we need an insertion in order to define invariants. In \cite{CK}, we take a line bundle $L$ over $X$ and insert the Euler class of the tautological bundle $L^{[n]}$, whose fibre over $Z \in \Hilb^n(X)$ is $H^0(L|_Z)$.

When $X$ is a toric Calabi-Yau fourfold, analogous invariants can be defined by localization as first proposed in \cite{CL}. Let
$T \subset (\C^*)^4$
denote the subtorus preserving the Calabi-Yau volume form. It is easy to show that we have an equality of schemes \cite{CK}
$$
\Hilb^n(X)^{T} = \Hilb^n(X)^{(\C^*)^4}.
$$
For any $T$-equivariant line bundle $L$ on $X$, we define 
\begin{align*}
\DT_4(X,L,o(\lL),\,n):=& \, \sum_{[Z] \in \Hilb^n(X)^T} (-1)^{o(\lL|_Z)} \sqrt{\frac{(-1)^{n}\cdot e_{T}\big(\Ext^{2}(I_Z,I_Z)\big)}
{e_{T}\big(\Ext^{1}(I_Z,I_Z)\big)e_{T}\big(\Ext^{3}(I_Z,I_Z)\big)}} \cdot e_T(L^{[n]}|_Z)
\\
\in& \, \frac{\mathbb{Q}(\lambda_1, \lambda_2, \lambda_3,\lambda_4)}{(\lambda_1+\lambda_2+\lambda_3+\lambda_4)},
\end{align*}
where the signs $(-1)^{o(\lL|_Z)}=\pm\,1$ are the choice of orientation $o(\lL)$. 

In \cite{CK}, the authors conjectured a formula for their generating series: there exists a choice of orientation $o(\lL)$ on each $\Hilb^n(X)$ such that the following formula holds
$$
1+\sum_{n=1}^{\infty}\DT_4(X,L,o(\lL),\,n) \, q^n = M(-q)^{\int_X c^T_3(TX) \, c^T_1(L)},
$$
where $M(q)$ still denotes the MacMahon function. 
For $X = \C^4$, this conjecture is a special case of a more general $K$-theoretic conjecture of Nekrasov \cite{Nekrasov}. However, our conjecture makes sense on (and is motivated by) an analogous conjecture on projective Calabi-Yau fourfolds. 
The above conjecture is proved when 
(i) $L=\oO_X(D)$ corresponds to a smooth toric divisor $D \subset X$, (ii) $L$ is arbitrary and $n \leqslant 6$. 

\subsection*{Higher dimensions} 
Let $X = \C^d$ with $d\geqslant3$. We consider the standard torus action of $(\C^*)^d$ on $\C^d$ and denote the Calabi-Yau subtorus by $T$. As before, we denote the equivariant parameters by $\lambda_1, \ldots, \lambda_d$. 
As in the $d=4$ case, we have
$$\Hilb^n(\C^d)^T = \Hilb^n(\C^d)^{(\C^*)^d}$$ 
as schemes (Lemma \ref{T fixed pts are iso}). Moreover, this scheme consists of finitely many reduced points corresponding to $(d-1)$-partitions\footnote{2-partitions are commonly known as plane partitions and 3-partitions are known as solid partitions. We recall the definitions in Section \ref{higher dim partition}.} $\pi = \{\pi_{i_1, \ldots, i_{d-1}}\}_{i_1, \ldots, i_{d-1} \geqslant 1}$ of size 
$$
|\pi| := \sum_{i_1, \ldots, i_{d-1} \geqslant 1} \pi_{i_1, \ldots, i_{d-1}} =  n.
$$

For $d \not \equiv 2\ \mathrm{mod}\, 4$, we consider the following formal analogues of the previous invariants
\begin{align*}
Z_n(\C^d) :=& \, \sum_{[Z] \in \Hilb^n(X)^{(\C^*)^d}} \frac{e_{(\C^*)^d}(\Ext^{\mathrm{even}}(I_Z,I_Z)_0)}{e_{(\C^*)^d}(\Ext^{\mathrm{odd}}(I_Z,I_Z)_0)}, \quad \mathrm{for \ } d \ \mathrm{odd} \\
Z_n(\C^d,L,o(\lL)) :=& \, \sum_{[Z] \in \Hilb^n(X)^T} (-1)^{o(\lL|_Z)} \sqrt{(-1)^{n}\frac{e_{T}\big(\Ext^{\mathrm{even}}(I_Z,I_Z)_0\big)}{e_{T}\big(\Ext^{\mathrm{odd}}(I_Z,I_Z)_0\big)}} \cdot e_T(L^{[n]}|_Z), \,d \equiv 0\ \mathrm{mod} \, 4.
\end{align*}
In the second case, $L$ is a $T$-equivariant line bundle on $X$. Furthermore, $(-1)^{o(\lL|_Z)}$ corresponds to a choice of sign for the square root, which we refer to as a choice of orientation. Moreover, $\Ext^{\mathrm{even}}(I_Z,I_Z)_0$ is short-hand for the sum of the trace-free extension groups $\Ext^{2i}(I_Z,I_Z)_0$ in the equivariant $K$-theory $K_T(\bullet)$ of one point and similarly for $\Ext^{\mathrm{odd}}(I_Z,I_Z)_0$. 
For this to be well-defined, we conjecture that (Conjecture ~\ref{key conj}) 
$$
\Ext^{\mathrm{even}}(I_Z,I_Z)_0 - \Ext^{\mathrm{odd}}(I_Z,I_Z)_0,
$$
has no negative $T$-fixed term, which we verify for all odd dimensions (Proposition \ref{verify key conj}, following \cite{MNOP}) and for $d \equiv 0\ \mathrm{mod}\, 4$ in all cases that we calculate (Theorem \ref{checks}). 

Denote the generating function of $d$-partitions by
$$
M_d(q) :=\sum_{{\scriptsize{\begin{array}{c} d\textrm{-partitions} \, \pi \end{array}}}} q^{|\pi|}. 
$$
Famously, MacMahon showed $M_2(q) = M(q)$ (given above), but no closed product formula is known for $M_{\geqslant 3}(q)$. 

Our first observation is that there is no straight-forward generalization of the conjectures in dimensions 3 and 4 to higher dimensions:
\begin{prop} [Remarks \ref{remfail1}, \ref{remfail2}]
Let $d=5$ or $7$. Then for any $E \in \mathbb{Q}(\lambda_1, \ldots, \lambda_d)$, we have
\begin{align*}
1+\sum_{n=1}^{\infty} Z_n(\C^d) \, q^n \not= M_{d-1}(-q)^E. 
\end{align*}
Let $d=8$ or $12$. Then there are $(\mathbb{C}^*)^d$-equivariant line bundles $L$ on $\C^d$ such that for any choice of orientation $o(\lL)$ and any $E \in \mathbb{Q}(\lambda_1, \ldots, \lambda_d) / (\lambda_1+ \cdots + \lambda_d)$, we have
\begin{align*}
1+\sum_{n=1}^{\infty} Z_n(\C^d,L,o(\lL)) \, q^n \not= M_{d-2}(-q)^E. 
\end{align*}
\end{prop}

\begin{rmk}
This result could be viewed as an indication that dimensions 3 and 4 are special. 
Perhaps it is also related to Nekrasov's comment in his paper ``Magnificent four'' \cite{Nekrasov}: \\

\noindent \textit{``The adjective `Magnificent' reflects this author's conviction that the dimension four is the maximal dimension where the natural albeit complex-valued probability distribution exists.''}
\end{rmk}

Notwithstanding, it seems that part of the formulae in dimensions 3 and 4 survives in higher dimensions:
\begin{thm}[Theorem \ref{thm for toric 2k-1}]\label{thm for toric 2k-1 intro}
For any $d \geqslant 3$ odd, we have
\begin{equation}
1+\sum_{n=1}^{\infty} Z_n(\C^d) \, q^n  \Big|_{\lambda_1 + \cdots +\lambda_d = 0} = M_{d-1}(-q).  
\nonumber \end{equation}
\end{thm}
In fact, the proof of Theorem \ref{thm for toric 2k-1 intro} easily follows by arguments analogous to \cite{MNOP} (see Section \ref{pf of odd conj}). \\

As for dimensions multiple of four, we conjecture the following.
\begin{conj} \label{conj for toric 4k intro}
Let $d \geqslant 4$ such that $d \equiv 0 \ \mathrm{mod}\, 4$ and let $\ell \in \Z$. Denote by $L = \oO_{\mathbb{C}^d} \otimes t_d^{-\ell}$ the trivial line bundle with character $t_d^{-\ell}$. Then for any choice of orientation $o(\lL)$ and any $[Z] \in \Hilb^n(\mathbb{C}^d)^T$ corresponding to a $(d-1)$-partition $\pi$, we have
\begin{equation} \label{spec}
(-1)^{o(\lL)|_Z} \sqrt{(-1)^{n}\frac{e_{T}\big(\Ext^{\mathrm{even}}(I_Z,I_Z)_0\big)}{e_{T}\big(\Ext^{\mathrm{odd}}(I_Z,I_Z)_0\big)}} \cdot e_T(L^{[n]}|_Z) \Big|_{\lambda_1+ \cdots + \lambda_{d-1} = \lambda_d = 0} = (-1)^{|\pi|} \omega_\pi \prod_{i=1}^{\pi_{1\ldots1}} (\ell-(i-1)),
\end{equation}
for some $\omega_\pi \in \Q$. Moreover, there exists a choice of orientation $o(\lL)$ on each $\Hilb^n({\mathbb{C}^d})$ such that
\begin{equation}\label{mainformula4k}
1+\sum_{n=1}^{\infty} Z_n(\C^d,L,o(\lL)) \, q^n  \Big|_{\lambda_1+ \cdots + \lambda_{d-1}  = \lambda_{d} = 0} = M_{d-2}(-q)^{\ell}.  
\end{equation}
\end{conj}
The above choice of orientation seems to be unique. 
\begin{conj} \label{sign unique intro}
The choices of orientation for which \eqref{mainformula4k} holds are unique. Specifically, they are the choices of orientation for which $\omega_\pi > 0$ in equation \eqref{spec}.
\end{conj}
\begin{thm}[Theorem \ref{main thm on 4k dim}, Proposition \ref{more verify of conj 4k}, \ref{verify sign unique}, \ref{weight wpi check}] \label{checks}
The conjectures hold in the following cases:
\begin{itemize}  
\item Conjecture \ref{conj for toric 4k intro} is true modulo $q^2$. 
\item Conjecture \ref{conj for toric 4k intro} is true for $\ell = 1$. 
\item Conjectures \ref{conj for toric 4k intro} and \ref{sign unique intro} are true in dimension 4 modulo $q^7$ \cite{CK}.  
\item Conjectures \ref{conj for toric 4k intro} and \ref{sign unique intro} are true in dimension 8 modulo $q^7$.  
\item Conjectures \ref{conj for toric 4k intro} and \ref{sign unique intro}  are true in dimension 12 modulo $q^5$. 
\item Equation \eqref{spec} is true for a certain list 3-partitions of size 7--15 \cite[App.~A]{CK}. Equation \eqref{spec} is true for the $7$-partitions of size $9,10,14$ of Remark \ref{indivpart}.
\end{itemize}
\end{thm}

\begin{rmk}
Donaldson-Thomas theory has deep relations with shifted symplectic geometry introduced by T.~Pantev, B.~T\"{o}en, M.~Vaqui\'{e}, and G.~Vezzosi \cite{PTVV}. While the former is only known to exist in dimensions 3 and 4, the latter exists in all dimensions. 
It would be interesting to understand whether the counting invariants studied in this paper 
are related to PTVV's program of quantization of Calabi-Yau moduli spaces. 
\end{rmk}
Assuming Conjectures \ref{conj for toric 4k intro} and \ref{sign unique intro}, we have the following application to counting \textit{weighted} $(d-1)$-partitions.
\begin{thm}[Theorem \ref{4k-1 partition counting}]
Let $d \geqslant 4$ such that $d \equiv 0\ \mathrm{mod}\, 4$. Conjecture \ref{conj for toric 4k intro} implies the following formula \begin{equation} \label{weightedcount}
\sum_{{\scriptsize{\begin{array}{c} (d-1)\textrm{-}\mathrm{partitions} \, \pi \end{array}}}} \omega_\pi \, t^{\pi_{1 \ldots 1}} \, q^{|\pi|} = e^{t(M_{d-2}(q)-1)},
\end{equation}
where $t$ is a formal variable. In particular, by setting $t = 1$, we obtain
\begin{equation}
\sum_{(d-1)\textrm{-}\mathrm{partitions} \, \pi} \omega_\pi \, q^{|\pi|} = e^{M_{d-2}(q)-1}. \nonumber 
\end{equation}
\end{thm}
In Definition \ref{combinatoric wpi}, we define a (purely combinatorial) weight $\omega_\pi^c$, which is conjecturally equal to $\omega_\pi$ (Conjecture \ref{compare wpi}). We check the equality $\omega_\pi^c = \omega_\pi$ in several cases (Proposition \ref{prop compare wpi}). Replacing $\omega_\pi$ by $\omega_\pi^c$, we prove \eqref{weightedcount} in Proposition \ref{comb proof} (in which case the requirement $d \equiv 0 \ \mathrm{mod} \, 4$ can be dropped).  

\subsection*{Acknowledgements} 
We are very grateful to the referee for helpful comments.
Y.~C. is supported by The Royal Society Newton International Fellowship. This material is based upon work supported by the National Science Foundation under Grant No.~DMS-1440140 while M. K. was in residence at the Mathematical Sciences Research Institute in Berkeley, California, during the Spring 2018 Semester.

\section{Hilbert schemes and tautological bundles}

\subsection{Definitions}

Let $X$ be a smooth variety. We denote by $\Hilb^{n}(X)$ the Hilbert scheme of $n$ points on $X$. This is a fine moduli space (i.e.~with universal family), whose closed points correspond to zero-dimensional subschemes of length $n$ in $X$. Moreover, when $b_1(X) = 0$, $\Hilb^{n}(X)$ is isomorphic to the moduli scheme of Gieseker stable sheaves on $X$ with Chern character $(1,0,\ldots,0,-n)\in H^{\mathrm{even}}(X)$ (basically, by mapping $Z$ to its defining ideal $I_Z \subset \oO_X$).

Given a line bundle $L$ on $X$, we define its tautological bundle $L^{[n]}$ as follows \cite[Sect.~4.1]{Lehn}
\begin{equation}L^{[n]}:=\pi_{M*} \big(\mathcal{O}_{\mathcal{Z}_n}\otimes \pi_X^{*}L\big),   \nonumber \end{equation}
where $\mathcal{Z}_n\subset \Hilb^{n}(X)\times X$ denotes the universal subscheme and 
$\pi_M, \pi_X$ are projections from the product $\Hilb^{n}(X)\times X$ to each factor. Since $\pi_M$ is a flat finite morphism of degree $n$, $L^{[n]}$ is locally free of rank $n$ on $\Hilb^{n}(X)$. Note that for any $[Z] \in\Hilb^n(X)$, we have 
\begin{align*}\chi(I_Z,I_Z)_0&=\chi(I_Z,I_Z)-\chi(\oO_X)=-2n, \quad \textrm{ }\textrm{if}\textrm{ }\dim_{\mathbb{C}}(X)\textrm{ }\textrm{is}\textrm{ }\textrm{even}, \nonumber  \\\chi(I_Z,I_Z)_0&=\chi(I_Z,I_Z)-\chi(\oO_X)=0, \quad \textrm{ }\textrm{if}\textrm{ }\dim_{\mathbb{C}}(X)\textrm{ }\textrm{is}\textrm{ }\textrm{odd}. \nonumber \end{align*}
 
When $L$ corresponds to an effective divisor $D$, the vector bundle $L^{[n]}$ has a tautological section, induced by the defining section $s_D : \oO \rightarrow \oO(D)$,
whose zero locus is the Hilbert scheme $\Hilb^n(D)$ of $n$ points on $D$ (e.g.~see \cite[Prop.~2.4]{CK} for a proof). 
\begin{prop} \label{section s}
Let $D \subset X$ be an effective divisor on a smooth variety $X$ and let $L:=\oO(D)$. Then the vector bundle $L^{[n]}$ on $\Hilb^n(X)$ has a tautological section whose (scheme theoretic) zero locus is isomorphic to $\Hilb^n(D)$.
\end{prop}

\subsection{Heuristics}
Let $X$ be a smooth projective variety of dimension $d:= \dim_{\mathbb{C}}(X)\geqslant 2$. In analogy with dimensions 3 and 4, we would like to treat $-\chi(I_Z,I_Z)_0 = \chi(\oO_X) - \chi(I_Z,I_Z)$ as the (real/complex) virtual dimension of $\Hilb^n(X)$. Namely, 
we imagine there exist
virtual classes
\begin{align*}
[\Hilb^n(X)]^{\mathrm{vir}} &\in H_{-2\chi(I_Z,I_Z)_0}(\Hilb^n(X)) =  H_{0}(\Hilb^n(X)), \textrm{ }\textrm{if}\textrm{ }
d \ \mathrm{odd}, \\
[\Hilb^n(X)]^{\mathrm{vir}} &\in H_{-\chi(I_Z,I_Z)_0}(\Hilb^n(X)) = H_{2n}(\Hilb^n(X)), \textrm{ }\textrm{if}\textrm{ }
d \equiv 0 \ \mathrm{mod}\, 4, \\
[\Hilb^n(X)]^{\mathrm{vir}} &\in H_{-2\chi(I_Z,I_Z)_0}(\Hilb^n(X)) = H_{4n}(\Hilb^n(X)), \textrm{ }\textrm{if}\textrm{ }
d \equiv 2\ \mathrm{mod}\, 4. 
\end{align*}

In the first case, we define invariants by simply taking the degree\,\footnote{When $d=3$, the required virtual class exists by \cite{MNOP, Thomas}.}
\begin{equation}\int_{[\Hilb^n(X)]^{\mathrm{vir}}}1.  \nonumber \end{equation}
In the second case, we integrate against $e(L^{[n]})$, where $L=\oO_X(D)$ corresponds to a smooth effective divisor. And we expect
\begin{equation}\int_{[\Hilb^n(X)]^{\mathrm{vir}}}e(L^{[n]})=\int_{[\Hilb^n(D)]^{\mathrm{vir}}}1. \nonumber \end{equation}
By Proposition \ref{section s}, this equality would be justified if $\Hilb^n(X)$ and $\Hilb^n(D)$ were smooth \emph{and} of dimension equal to the virtual dimension (which, of course, never happens).\,\footnote{When $d=4$ and $X$ is Calabi-Yau, the virtual class exists (see \cite{BJ, CL}). It depends on a choice of orientation.}

In the third case, one can use many possible insertions, e.g.~Chern classes of tautological bundles and virtual tangent bundles, etc. We do not consider them in this paper.

Of course, when $d>4$, we do not know how to construct virtual classes on $\Hilb^n(X)$, or even whether they exist at all.
Nevertheless, on a toric variety, such as $X=\mathbb{C}^d$, we can define an equivariant version of $[\Hilb^n(X)]^{\mathrm{vir}}$ in analogy with the virtual localization formulae for $d=3,4$.

\subsection{Deformation and obstruction spaces}
In this section, we gather some facts about the deformation-obstruction spaces of $\Hilb^n(X)$ for later use.

The following proposition closely follows  \cite[Lemma 2.6]{CK}. We include its proof for completeness.
\begin{prop}\label{ideal and str sheaf}
Let $X$ be a smooth quasi-projective toric variety\,\footnote{A toric variety is defined by a fan $\Delta$ in a lattice $N$ \cite{Ful}. We always assume $\Delta$ contains cones of dimension $\rk(N)$.} with torus $(\mathbb{C}^*)^d$ and satisfying $H^{i>0}(\oO_X)=0$. Then for any $[Z]\in\Hilb^n(X)^{(\mathbb{C}^*)^d}$ and $0<i< d:=\dim_{\mathbb{C}}X$, we have an isomorphism of $(\mathbb{C}^*)^d$-representations
\begin{equation}\Ext^{i}(I_{Z},I_{Z}) \cong \Ext^{i}(\mathcal{O}_Z,\mathcal{O}_Z). \nonumber \end{equation} 
Moreover $\Ext^d(I_Z,I_Z) = 0$.
\end{prop}
\begin{proof}
Suppose $d\geqslant 2$ (or else the statement of the proposition is empty). Since $H^{i>0}(\oO_X)=0$, 
\begin{align*}\Ext^{i}(I_{Z},I_{Z})\cong \Ext^{i}(I_{Z},I_{Z})_0 \quad\textrm{for}\,\, i>0.  \end{align*} 
Applying $\RHom(-,\oO_Z)$ to the short exact sequence 
\begin{equation} \label{idealseq} 0\to I_Z\to \oO_X \to \oO_Z\to 0,   \end{equation}
we obtain isomorphisms
\begin{equation}\label{map i2}\Ext^i(I_Z,\oO_Z) \cong \Ext^{i+1}(\oO_Z,\oO_Z), \quad i\geqslant0, \end{equation}
where we use $H^{i>0}(\oO_Z) = 0$ and $\Hom(\oO_Z,\oO_Z) \cong \Hom(\oO_X,\oO_Z)$. 

Furthermore, we have the following diagram 
\begin{align*}
\xymatrix{    &  \dR \Gamma(\oO_X)[1] \ar@{=}[r]\ar[d] &
\dR \Gamma(\oO_X)[1] \ar[d] \\
\RHom(I_{Z}, \mathcal{O}_Z) \ar[r] & \RHom(I_{Z}, I_{Z})[1] \ar[r] \ar[d] &
\RHom(I_{Z}, \oO_X)[1] \ar[d] \\
&  \RHom(I_{Z}, I_{Z})_0[1]  &  \RHom(\mathcal{O}_{Z}, \oO_X)[2], }\end{align*}
where the horizontal and vertical arrows are distinguished triangles.
By taking cones, we obtain a distinguished triangle
\begin{equation} 
\RHom(I_{Z}, \mathcal{O}_Z) \to \RHom(I_{Z}, I_{Z})_0[1] \to
\RHom(\mathcal{O}_Z, \oO_X)[2].
\nonumber \end{equation}
Hence we have an exact sequence
\begin{equation} \cdots \to\Ext^{i}(I_{Z}, \mathcal{O}_Z )\to \Ext^{i+1}(I_{Z}, I_{Z} )_0 \to \Ext^{i+2}(\mathcal{O}_Z,\oO_X)
\to \Ext^{i+1}(I_{Z}, \mathcal{O}_Z )\to \cdots. \nonumber \end{equation}
Next, we note that $\Ext^{i}(\mathcal{O}_Z,\oO_X)=0$ for $i < d$. This follows from the fact that $\mathcal{E}xt^{i<d}(\oO_Z,\oO_X)=0$  \cite[p.~78]{Huy} and the local-to-global spectral sequence $H^p(X,\mathcal{E}xt^{q}(-,-))\Rightarrow \Ext^{p+q}(-,-)$ \cite[p.~85, (3.16)]{Huy}. Consequently, we have 
\begin{equation}\Ext^{i}(I_{Z}, \mathcal{O}_Z )\cong\Ext^{i+1}(I_{Z}, I_{Z} )_0, \quad \textrm{ }\textrm{for all}\textrm{ }i\leqslant d-3. \nonumber \end{equation}
By applying $\RHom(-,\oO_X)$ to \eqref{idealseq}, we find
\begin{align}
\begin{split} \label{map i1}
\Hom(I_Z,\oO_X) &\cong \Hom(\oO_X,\oO_X), \\
\Ext^{d-1}(I_Z,\oO_X)&\cong \Ext^d(\oO_Z,\oO_X), \\
\Ext^i(I_Z,\oO_X)&= 0, \textrm{ }\textrm{for}\textrm{ } i\neq 0\textrm{ }\textrm{or}\textrm{ }d-1,
\end{split}
\end{align}
where we use $H^{i > 0}(\oO_X) = 0$ and $\Ext^i(\oO_Z,\oO_X) = 0$ for $i < d$. 

By applying $\RHom(I_Z,-)$ to \eqref{idealseq}, we obtain an exact sequence  
\begin{equation*}\label{Trepsseq}\cdots \to \Ext^i(I_Z,\oO_Z)\to\Ext^{i+1}(I_Z,I_Z)\to \Ext^{i+1}(I_Z,\oO_X)\to \cdots.  \end{equation*}
Combining with (\ref{map i1}), we get the following isomorphisms and exact sequence 
\begin{align*} 
\begin{split} \label{twists} 
&\Ext^i(I_Z,\oO_Z) \cong \Ext^{i+1}(I_Z,I_Z), \textrm{ }\textrm{for}\textrm{ }0\leqslant i\leqslant d-3, \\
&0 \to \Ext^{d-2}(I_Z,\oO_Z)\to\Ext^{d-1}(I_Z,I_Z)\to \Ext^{d-1}(I_Z,\oO_X)\stackrel{\eta\,}{\to} \Ext^{d-1}(I_Z,\oO_Z)  \\ 
&\ \, \to\Ext^{d}(I_Z,I_Z)\to\Ext^{d}(I_Z,\oO_X)=0. 
\end{split}
\end{align*}

We claim that the map $\eta$ is an isomorphism, so $\Ext^{d-2}(I_Z,\oO_Z)\cong\Ext^{d-1}(I_Z,I_Z)$. In fact, we have a commutative diagram
\begin{equation}
\xymatrix{\ar @{} [dr] |{} \Hom(I_Z,\oO_X[d-1]) \ar[d]^{i_1} \ar[r]^{\eta} & \Hom(I_Z,\oO_Z[d-1]) \ar[d]^{i_2} 
&  \\  \Hom(\oO_Z[-1],\oO_X[d-1]) \ar[r]^{\phi}
&\Hom(\oO_Z[-1],\oO_Z[d-1]), }
\nonumber \end{equation}
where $i_1$, $i_2$ are isomorphisms in (\ref{map i1}), (\ref{map i2}) respectively, and $\phi$ is the map in the exact sequence 
\begin{equation}\to \Ext^d(\oO_Z,I_Z)\to \Ext^d(\oO_Z,\oO_X)\stackrel{\phi\,}{\to} \Ext^d(\oO_Z,\oO_Z)\to 0,  \nonumber \end{equation}
obtained by applying $\RHom(\oO_Z,-)$ to (\ref{idealseq}). 

By Hirzebruch-Riemann-Roch and Serre duality\,\footnote{HRR and Serre duality hold also here as $Z$ is compactly supported, see 
e.g. \cite[footnote 12]{CK}.}, we have
\begin{align*}
\dim_{\mathbb{C}}\Ext^d(\oO_Z,\oO_X)&= \dim_{\mathbb{C}} H^0(X,\mathcal{E}{\it{xt}}^d(\oO_Z,\oO_X)) \\
&= \chi(\oO_Z,\oO_X)=n, \nonumber  \\
\dim_{\mathbb{C}}\Ext^d(\oO_Z,\oO_Z)&=\dim_{\mathbb{C}}\Ext^0(\oO_Z,\oO_Z)=n. \nonumber 
\end{align*} 
Therefore $\phi$ is an isomorphism and
so is $\eta$.
\end{proof}
In this paper, a toric Calabi-Yau variety is defined to be a smooth quasi-projective toric variety satisfying $K_X \cong \oO_X$ and $H^{i>0}(\oO_X)=0$. In the even-dimensional case we have the following.
\begin{prop}\label{compare ext}
Let $X$ be a toric Calabi-Yau variety of dimension $\dim_{\mathbb{C}}X=2k >0$ and denote the Calabi-Yau torus by $T$.  Let $D\subset X$ be a smooth $T$-invariant divisor. 
Then for any $T$-invariant zero-dimensional subscheme $Z\subset D\subset X$, we have the following equality in $K_T(\bullet)$
\begin{equation} \sum_{i=1}^{2k-1}(-1)^{i-1}\Ext^i_X(\iota_*\oO_Z,\iota_*\oO_Z)=
\sum_{i=1}^{2k-1}(-1)^{i-1}\big(\Ext^i_D(\oO_Z,\oO_Z)+\Ext^i_D(\oO_Z,\oO_Z)^*\big), \nonumber \end{equation}
where $\iota : D \subset X$ denotes the inclusion.
\end{prop}
\begin{proof}
The inclusion $\iota:D\hookrightarrow X$ gives a distinguished triangle \cite[Cor.~11.4, p.~249]{Huy}
\begin{equation}\oO_Z\otimes N_{D/X}^{-1}[1] \to \mathbf{L} \iota^*\iota_*\oO_Z\to \oO_Z,   \nonumber \end{equation}
where $N_{D/X} \cong \oO_D(D)$ denotes the normal bundle of $D$ in $X$. Since $X$ is a toric Calabi-Yau variety and $T$ is the Calabi-Yau torus, adjunction gives a \emph{$T$-equivariant} isomorphism $\oO_D(D) \cong K_D$. We obtain a distinguished triangle 
\begin{equation}\RHom_D(\oO_Z,\oO_Z) \to \RHom_X(\iota_*\oO_Z,\iota_*\oO_Z) \to \RHom_D(\oO_Z,\oO_Z\otimes K_D)[-1].  \nonumber \end{equation}
By $T$-equivariant Serre duality, it becomes 
\begin{equation}\RHom_D(\oO_Z,\oO_Z) \to \RHom_X(\iota_*\oO_Z,\iota_*\oO_Z) \to \RHom_D(\oO_Z,\oO_Z)^{*}[-2k].  \nonumber \end{equation}
Its cohomology gives an exact sequence 
\begin{align*}0\to \Ext^1_D(\oO_Z,\oO_Z)&\to \Ext^1_X(\iota_*\oO_Z,\iota_*\oO_Z)\to \Ext_D^{2k-1}(\oO_Z,\oO_Z)^{*}  \nonumber \\
&\to \Ext^2_D(\oO_Z,\oO_Z)\to \Ext^2_X(\iota_*\oO_Z,\iota_*\oO_Z)\to \Ext_D^{2k-2}(\oO_Z,\oO_Z)^{*}\to\cdots,  \nonumber \end{align*}
which implies
\begin{equation*} \sum_{i=1}^{2k-1}(-1)^{i-1}\Ext^i_X(\iota_*\oO_Z,\iota_*\oO_Z)=
\sum_{i=1}^{2k-1}(-1)^{i-1}\big(\Ext^i_D(\oO_Z,\oO_Z)+\Ext^i_D(\oO_Z,\oO_Z)^*\big). \qedhere \end{equation*}
\end{proof}

\subsection{Monomial ideals and partitions}\label{higher dim partition} 

We recall the definition of an $n$-partition (most commonly studied for $n=1,2,3$).
\begin{defi}
An \textit{n-partition}
$\pi=\{\pi_{i_1\ldots i_n}\}_{i_1,\ldots, i_n \geqslant 1}$ consists of a sequence of non-negative integers 
$\pi_{i_1\ldots i_n}$ satisfying
\begin{align*}
&\pi_{i_1\ldots i_n}\geqslant \pi_{j_1\ldots j_n}\,, \textrm{ }\textrm{whenever}\textrm{ }1\leqslant i_1\leqslant j_1\,,\,\ldots,\,1\leqslant i_n\leqslant j_n, \end{align*}
and
\begin{equation}|\pi| := \sum_{i_1,\ldots, i_n  \geqslant 1} \pi_{i_1\ldots i_n}< \infty. \nonumber \end{equation}
Here $|\pi|$ denotes the \textit{size} of $\pi$. For example, ~1-partitions are commonly known as \emph{partitions}, 2-partitions as \emph{plane partitions}, and 3-partitions as \emph{solid partitions}.
\end{defi}

Let $X=\mathbb{C}^d$ and let $(\mathbb{C}^*)^d$ act on $X$ by the standard action
\begin{equation}(t_1,\ldots,t_d)\cdot (x_1,\ldots,x_d):=(t_1 x_1,\ldots,t_d x_d). \nonumber \end{equation}
A $(\mathbb{C}^*)^d$-invariant zero-dimensional subscheme $Z\in\Hilb^{n}(X)^{(\C^*)^d}$ is supported at the origin and is cut out by a monomial ideal. Therefore, elements of $\Hilb^{n}(X)^{(\C^*)^d}$ are in bijective correspondence with $(d-1)$-partitions. Concretely, a $(d-1)$-partition $\pi =\{\pi_{i_1\ldots i_{d-1}}\}_{i_1,\ldots, i_{d-1} \geqslant 1}$ corresponds to the zero-dimensional subscheme $Z_\pi$ defined by
$$
I_{Z_\pi} := \Big\langle x_1^{i_1-1} \cdots x_{d-1}^{i_{d-1}-1}x_{d}^{\pi_{i_1\ldots i_{d-1}}} \ | \ i_1,\ldots, i_{d-1} \geqslant 1 \,\Big\rangle,
$$
where $|\pi|$ equals the length of $Z_\pi$. The $(\mathbb{C}^*)^d$-equivariant representation of $Z_{\pi}$ is given by 
\begin{equation} \label{Zpi}
Z_\pi = \sum_{i_1,\ldots, i_{d-1} \geqslant 1} \sum_{m=1}^{\pi_{i_1\ldots i_{d-1}}} t_1^{i_1-1} \cdots t_{d-1}^{i_{d-1}-1} t_d^{m-1},
\end{equation}
where the sum is over all $i_1,\ldots,i_{d-1} \geqslant1$ for which $\pi_{i_1\ldots i_{d-1}}\geqslant 1$. 

\begin{rmk}
Fix $n \geqslant 1$ and define numbers $P_n(i)$ by 
$$
P_n(i):= \left\{\begin{array}{cc} \textrm{the}\textrm{ }\textrm{number}\textrm{ }\textrm{of}\textrm{ }n\textrm{-partitions}\textrm{ }\textrm{of}\textrm{ }\textrm{size}\textrm{ }i & \textrm{if  } n\geqslant1 \\ 
1 & \textrm{if}\textrm{ }\, n=0, \end{array}\right.
$$
and form the generating series 
\begin{equation*}\label{n-partition gene series}M_{n}(q):=\sum_{i=0}^{\infty}P_n(i)\,q^i. \end{equation*}
We have closed product formulae for $M_n(q)$ when $n=0,1,2$\,: 
\begin{align*}
M_{0}(q)&=\frac{1}{1-q}\quad \textrm{ }\textrm{geometric}\textrm{ }\textrm{series}, \\
M_{1}(q)&=\prod_{n=1}^{\infty}\frac{1}{(1-q^n)} \quad \textrm{ }\textrm{Euler}\textrm{ }\textrm{series},  \nonumber  \\
M_{2}(q)&=\prod_{n=1}^{\infty}\frac{1}{(1-q^n)^n} \quad \textrm{ }\textrm{McMahon}\textrm{ }\textrm{function}. \nonumber 
\end{align*}
No product formula is known (or known to exist) for $n>2$. 
For general $n \geqslant 1$, we have the following formula modulo $q^7$ (see e.g. \cite[Prop. 5.3]{Cheah})
\begin{align*}
M_n(q)&=\,1+q+\Big(1+n\Big)\,q^2+\Bigg(1+2n+\binom{n}{2} \Bigg)\,q^3
+\Bigg(1+4n+4\binom{n}{2}+\binom{n}{3}\Bigg)\,q^4+ \nonumber \\
&\,\Bigg(1+6n+11\binom{n}{2}+7\binom{n}{3}+\binom{n}{4}\Bigg)\,q^5+
\Bigg(1+10n+27\binom{n}{2}+28\binom{n}{3}+11
\binom{n}{4}+\binom{n}{5}\Bigg)\,q^6. \nonumber 
\end{align*}
\end{rmk}

\section{Zero-dimensional counts on $\mathbb{C}^d$}

\subsection{Calabi-Yau torus}

Let $X=\mathbb{C}^d$ with the standard $(\mathbb{C}^*)^d$-action
\begin{equation*}
t \cdot x_i = t_i x_i.
\end{equation*}
Let $\bullet$ be $\Spec \mathbb{C}$ with trivial $(\mathbb{C}^*)^d$-action. Denote by $\mathbb{C} \otimes t_i$ the 1-dimensional 
$(\mathbb{C}^*)^d$-representation with character $t_i$ and write $\lambda_i \in H_{(\mathbb{C}^*)^d}^{\ast}(\bullet)$ for its 
$(\mathbb{C}^*)^d$-equivariant first Chern class. Then
\begin{align*}H_{(\mathbb{C}^*)^d}(\bullet)=\mathbb{Z}[\lambda_1, \ldots,\lambda_d]. \end{align*} 
Let $T\subset (\mathbb{C}^*)^d$ be the Calabi-Yau subtorus
\begin{equation}T:=\Big\{(t_1,\ldots,t_d)\,\big{|}\,t_1 \cdots t_d=1 \Big\}. \nonumber \end{equation}
Then  
\begin{align*}
H_{T}(\bullet)=\frac{\mathbb{Z}[\lambda_1, \ldots,\lambda_d]}{\big(\sum_{i=1}^d\lambda_i\big)}
\cong \mathbb{Z}[\lambda_1,\ldots,\lambda_{d-1}]. 
\end{align*} 
The action of $(\mathbb{C}^*)^d$ lifts to the Hilbert scheme $\Hilb^n(\mathbb{C}^d)$, whose fixed points are labelled by $(d-1)$-partitions (Section \ref{higher dim partition}). The Calabi-Yau subtorus $T$ acts on $\Hilb^n(\mathbb{C}^d)$ and preserves the Serre duality pairing on Ext groups.
Just like in \cite[Lemmas 3.1, 3.4]{CK}, one easily proves: 
\begin{lem}\label{T fixed pts are iso}
We have an equality of schemes
\begin{equation}\Hilb^n(\mathbb{C}^d)^{(\mathbb{C}^*)^d}=\Hilb^n(\mathbb{C}^d)^{T} \nonumber \end{equation}
consisting of finitely many reduced points.
\end{lem}
In order to be able to define our counts in the next section, we make the following conjecture. 
\begin{conj}\label{key conj}
Let $d > 0$. For any torus fixed point $[Z]\in\Hilb^n(\mathbb{C}^d)^{T}$,
\begin{equation}\sum_i(-1)^i\Ext^i(I_Z,I_Z)_0\in K_T(\bullet) \nonumber \end{equation}
has no negative non-zero $T$-constant term.
\end{conj}
\begin{prop}\label{verify key conj}
Conjecture \ref{key conj} holds in the following cases: 
\begin{itemize}
\item When $d$ is odd.
\item When $d=2,4$.
\item When  $d \equiv 0\ \mathrm{mod}\, 4$ and $Z_\pi$ satisfies $|\pi| = 1$.
\item When $d=8$, for all $Z_\pi$ satisfying $|\pi| < 7$.
\item For $Z_\pi$ where $\pi$ is one of the 7-partitions of size 9, 10, 14 listed in Remark \ref{indivpart}.
\item When $d=12$, for all $Z_\pi$ satisfying $|\pi| < 5$.
\end{itemize}
\end{prop}
\begin{proof}
When $d$ is odd, any term $a\cdot t_1^{m_1}\cdots t_d^{m_d}$ in
$\chi(I_Z,I_Z)_0$ comes with a Serre dual term $-a\cdot t_1^{-m_1}t_2^{-m_2}\cdots t_d^{-m_d}$. Therefore all $T$-constant terms cancel pairwise.
The case $d=4$ is proved in \cite[Lemma 3.4]{CK} (the proof also works for $d=2$). The case $d \equiv 0\ \mathrm{mod}\, 4$ and $|\pi| = 1$ is part of the proof of Example \ref{verification of one pt for 4k}. For all other cases, we calculcated $\sum_i(-1)^i\Ext^i(I_Z,I_Z)_0$ explicitly by the vertex formalism of Section \ref{sect vertex formu} and observe that there are no negative $T$-constant terms.
\end{proof}

\subsection{Definitions}
We start with $d > 0$ arbitrary. Let $T\subset (\mathbb{C}^*)^d$ be the Calabi-Yau subtorus. 
For $[Z]\in \Hilb^n(X)^T$, we can form the complex vector bundle
\begin{equation}
\begin{array}{lll}
      & \quad ET\times_{T}\Ext^{i}(I_Z,I_Z)
      \\  &  \quad \quad\quad \quad \downarrow \\   &   \quad  ET\times_{T}\{I_Z\}=BT,
\end{array} \nonumber\end{equation}
whose Euler class is the $T$-equivariant Euler class $e_{T}\big(\Ext^{i}(I_Z,I_Z)\big)$. 
We collect all odd, (resp.~even, degree terms of $\Ext^{*}(I_Z,I_Z)$ and define 
\begin{equation}e_{T}\big(\Ext^{\mathrm{odd}}(I_Z,I_Z)\big):=\prod_{i=0}^{\infty} e_{T}\big(\Ext^{2i+1}(I_Z,I_Z)\big)\in H_{T}(\bullet), \nonumber \end{equation}
\begin{equation}e_{T}\big(\Ext^{\mathrm{even}}(I_Z,I_Z)\big):=\prod_{i=0}^{\infty} e_{T}\big(\Ext^{2i+2}(I_Z,I_Z)\big)\in H_{T}(\bullet). \nonumber \end{equation}
By Conjecture \ref{key conj}, the following quotient is well-defined
\begin{equation}\frac{e_{T}\big(\Ext^{\mathrm{even}}(I_Z,I_Z)\big)}{e_{T}\big(\Ext^{\mathrm{odd}}(I_Z,I_Z)\big)} 
\in \frac{\mathbb{Q}(\lambda_1, \cdots,\lambda_{d})}{(\lambda_1 + \cdots + \lambda_d)}.\nonumber \end{equation}

\subsection*{Odd dimensions}
\begin{defi}\label{def of equi invs for 2k-1}
Let $d \geqslant 3$ be an odd integer. Then $\chi(I_Z,I_Z)_0=0$ for all $[Z] \in \Hilb^n(\mathbb{C}^d)$, and we define
\begin{equation}Z_{\mathbb{C}^d}(q):=1+\sum_{n=1}^{\infty}\,\sum_{[Z]\in \Hilb^n(\mathbb{C}^d)^T}\frac{e_T\big(\Ext^{\mathrm{even}}(I_Z,I_Z)\big)}{e_T\big(\Ext^{\mathrm{odd}}(I_Z,I_Z)\big)}\cdot q^n \, \in \frac{\mathbb{Q}(\lambda_1, \cdots,\lambda_{d})}{(\lambda_1 + \cdots + \lambda_d)}[\![q]\!]. \nonumber \end{equation}
\end{defi}
\begin{exam}\label{eg 1}
Let $d > 1$ and let $I_x$ be the ideal sheaf of a closed point  $x\in (\mathbb{C}^d)^T$. Then the Koszul resolution of $x$ gives a $T$-equivariant isomorphism
\begin{equation}\Ext^i(I_x,I_x)\cong \bigwedge^iT_x\mathbb{C}^d, \quad \textrm{for all } 0< i< d. \nonumber \end{equation}
If $d$ is odd, then the coefficient of $q$ in $Z_{\mathbb{C}^d}(q)$ equals 
\begin{equation}
\prod_{i=1}^{\frac{d-1}{2}}
\frac{e_T\big(\bigwedge^{2i}T_x\mathbb{C}^d\big)}{e_T\big(\bigwedge^{2i-1}T_x\mathbb{C}^d\big)}=-1, 
\nonumber \end{equation}
where we use $T$-equivariant Serre duality (see Thm. \ref{thm for toric 2k-1} for the proof in a more general case).
\end{exam}

\subsection*{Dimensions multiple of four} Let $d = 4k >0$. Then $\chi(I_Z,I_Z)_0=2n$ for all $[Z]\in \Hilb^n(\mathbb{C}^d)$, so we need an insertion.
By Proposition \ref{ideal and str sheaf}, which allows us to use $T$-equivariant Serre duality for compactly supported sheaves, we have $T$-equivariant isomorphisms 
\begin{equation}\label{equ1}\Ext^{i}(I_Z,I_Z)\cong \Ext^{4k-i}(I_Z,I_Z)^* \textrm{ }\textrm{for}\textrm{ }0<i<4k,  \end{equation} 
and equalities
\begin{equation}e_{T}\big(\Ext^{i}(I_Z,I_Z)\big)=(-1)^{\ext^{i}(I_Z,I_Z)}\cdot e_{T}\big(\Ext^{4k-i}(I_Z,I_Z)\big). \nonumber \end{equation}
Consequently
\begin{equation}\frac{e_{T}\big(\Ext^{\mathrm{even}}(I_Z,I_Z)\big)}{e_{T}\big(\Ext^{\mathrm{odd}}(I_Z,I_Z)\big)} 
=(-1)^{n+\frac{1}{2}\ext^{2k}(I_Z,I_Z)}\Bigg(\frac{\prod_{i=1}^{k-1}e_{T}\big(\Ext^{2i}(I_Z,I_Z)\big)}{\prod_{i=1}^{k}e_{T}\big(\Ext^{2i-1}(I_Z,I_Z)\big)}\Bigg)^2\cdot\, e_{T}\big(\Ext^{2k}(I_Z,I_Z)\big). \nonumber \end{equation}

When $i=2k$, the isomorphism (\ref{equ1}) defines a $T$-invariant non-degenerate quadratic form on $\Ext^{2k}(I_Z,I_Z)$, hence induces a quadratic form $Q$ on $ET\times_{T}\Ext^{2k}(I_Z,I_Z)$. We define 
\begin{equation}\label{half euler class equivariant}e_{T}\big(\Ext^{2k}(I_Z,I_Z),Q\big)\in\mathbb{Z}[\lambda_1,\ldots,\lambda_{d}] / (\lambda_1+ \cdots + \lambda_d) \end{equation}
to be the \textit{half Euler class} of the quadratic bundle\,\footnote{I.e.~vector bundle with a non-degenerate quadratic form.} 
$(ET\times_{T}\Ext^{2k}(I_Z,I_Z),Q)$, i.e.~the Euler class of its positive real form.\,\footnote{I.e.~half rank real subbundle on which $Q$ is real and positive definite.} This notion is well-defined up to a sign corresponding to the choice of orientation of its positive real form. Moreover,  (see \cite[(2.4)]{CK})
\begin{equation}e_{T}\big(\Ext^{2k}(I_Z,I_Z),Q\big)^2 =(-1)^{\frac{1}{2}\ext^{2k}(I_Z,I_Z)}\cdot e_{T}\big(\Ext^{2k}(I_Z,I_Z)\big).
\nonumber \end{equation}
We obtain
\begin{equation}\frac{e_{T}\big(\Ext^{\mathrm{even}}(I_Z,I_Z)\big)}{e_{T}\big(\Ext^{\mathrm{odd}}(I_Z,I_Z)\big)} 
=(-1)^{n}\Bigg(\frac{\prod_{i=1}^{k-1}e_{T}\big(\Ext^{2i}(I_Z,I_Z)\big)}{\prod_{i=1}^{k}e_{T}\big(\Ext^{2i-1}(I_Z,I_Z)\big)}\cdot e_{T}\big(\Ext^{2k}(I_Z,I_Z),Q\big)\Bigg)^2,  \nonumber \end{equation}
or equivalently
\begin{equation}\label{equ2} (-1)^{o({\mathcal{L}})|_Z}  \sqrt{(-1)^{n}\frac{e_{T}\big(\Ext^{\mathrm{even}}(I_Z,I_Z)\big)}{e_{T}\big(\Ext^{\mathrm{odd}}(I_Z,I_Z)\big)} }
= \frac{\prod_{i=1}^{k-1}e_{T}\big(\Ext^{2i}(I_Z,I_Z)\big)}{\prod_{i=1}^{k}
e_{T}\big(\Ext^{2i-1}(I_Z,I_Z)\big)} \cdot e_{T}\big(\Ext^{2k}(I_Z,I_Z),Q\big).   \end{equation}
Some comments on this equation. Here $\sqrt{\alpha^2}$ is by definition equal to $\pm \alpha$, which is only determined up to sign. When we write $(-1)^{o({\mathcal{L}})|_Z} \sqrt{\alpha^2}$ it means we \emph{choose} a sign for the square root. In the above equation, the choice of sign $(-1)^{o({\mathcal{L}})|_Z}$ is equivalent to the choice of orientation used in defining $e_{T}\big(\Ext^{2k}(I_Z,I_Z),Q\big)$. 

\begin{defi}\label{def of equi invs for 4k}
Let $d=4k>0$
and let $L$ be a $T$-equivariant line bundle on $\mathbb{C}^d$.  Assume Conjecture \ref{key conj} holds. Define
\begin{equation}Z_{\mathbb{C}^d,L,o({\mathcal{L}})}(q):=1+\sum_{n=1}^{\infty}\,\sum_{[Z]\in \Hilb^n(\mathbb{C}^d)^T}(-1)^{o({\mathcal{L}})|_Z}
\sqrt{(-1)^{n}\cdot\frac{e_{T}\big(\Ext^{\mathrm{even}}(I_Z,I_Z)\big)}{e_{T}\big(\Ext^{\mathrm{odd}}(I_Z,I_Z)\big)}}\cdot
e_{T}(L^{[n]}|_Z)\cdot q^n, \nonumber \end{equation}
which is an element in $\mathbb{Q}(\lambda_1, \ldots,\lambda_d) / (\lambda_1+ \cdots +\lambda_d)[\![q]\!]$.
Here 
$o(\lL)|_Z$ denotes a choice of sign for each  square root (\ref{equ2}). For $d=4$, this recovers the definition of \cite{CK}.
\end{defi}
\begin{exam}\label{exam 2}
Consider the ideal sheaf $I_x$ of $x\in (\mathbb{C}^d)^T$, where $d=4k>0$. Referring to Examples \ref{eg 1}, we deduce that
the coefficient of $q$ in $Z_{\mathbb{C}^d,L,o({\mathcal{L}})}(q)$ equals 
\begin{equation}
(-1)^{o({\mathcal{L}})|_x}\sqrt{(-1)\cdot\frac{\prod_{i=1}^{2k-1}e_{T}\big(\bigwedge^{2i}T_x\mathbb{C}^d\big)}{\prod_{i=1}^{2k}e_{T}\big(\bigwedge^{2i-1}T_x\mathbb{C}^d\big)}}\cdot e_{T}(L|_x).  \nonumber \end{equation}
\end{exam}

\subsection{Conjecture}

Let $d\geqslant 3$ be an odd integer. Recall the definition of $M_d(q)$ from Section \ref{higher dim partition}.

\begin{rmk} \label{remfail1}
Suppose we had defined $Z_{\mathbb{C}^d}(q)$ exactly as in Definition \ref{def of equi invs for 2k-1} \emph{but} using the full torus $(\mathbb{C}^*)^d$ instead of the Calabi-Yau torus $T$. Then $Z_{\mathbb{C}^d}(q)$ is an element of
$$
\Q(\lambda_1, \ldots, \lambda_d)[[q]].
$$
In analogy to Maulik-Nekrasov-Okounkov-Pandharipande's work in three dimensions \cite[Conj.~1]{MNOP}, one may expect
\begin{equation*}
Z_{\mathbb{C}^d}(q) \stackrel{?}{=} M_{d-1}(-q)^E 
\end{equation*}
for some $E \in \Q(\lambda_1, \ldots, \lambda_d)$. However, after developing a vertex formalism as in MNOP in Section \ref{calc}, we find that such an equation already fails $\textrm{mod}\, q^3$ in dimensions 5 and 7.
\end{rmk}

Nevertheless, staying close to the original arguments of \cite{MNOP}, it is easy to obtain the following (where we go back to using the Calabi-Yau torus as in Definition \ref{def of equi invs for 2k-1}). The proof is given in Section \ref{pf of odd conj}.

\begin{thm}\label{thm for toric 2k-1}
For all $d \geqslant 3$ odd, we have
\begin{equation}Z_{\mathbb{C}^d}(q)=M_{d-1}(-q).  \nonumber \end{equation}
\end{thm}

Next, let $d \geqslant 4$ such that $d \equiv 0\ \mathrm{mod}\, 4$.
\begin{rmk} \label{remfail2}
In \cite{CK} (see also \cite{Nekrasov} in a more general setting), we provide a conjectural formula for $Z_{\mathbb{C}^4,L,o({\mathcal{L}})}(q)$. Specifically, let $L  = \mathcal{O}_{\mathbb{C}^4} \otimes t_1^{d_1}t_2^{d_2}t_3^{d_3}t_4^{d_4}$, we conjectured (and provided evidence for) the following formula
$$
Z_{\mathbb{C}^4,L,o({\mathcal{L}})}(q) =  M(-q)^{\frac{\big(\sum_i d_i \lambda_i\big) \big(- \sum_{i<j<k} \lambda_i \lambda_j \lambda_k \big)}{\lambda_1\lambda_2\lambda_3\lambda_4}},
$$
where $\lambda_1 + \lambda_2 + \lambda_3 + \lambda_4 = 0$, $o(\mathcal{L})$ are appropriate choices of orientation, and $M(q) = M_2(q)$ is the MacMahon function. However, the direct analog of this conjecture is \emph{not} true in dimensions $d \geqslant 8$ satisfying $d \equiv 0\ \textrm{mod}\, 4$. Specifically, for dimensions $d = 8$, $12$, we checked that there are $T$-equivariant line bundles $L$ on $\mathbb{C}^d$ for which there do not exist orientations $o(\mathcal{L})$ such that
$$
Z_{\mathbb{C}^d,L,o({\mathcal{L}})}(q) =  M_{d-2}(-q)^{E},
$$
for any $E \in \Q(\lambda_1, \ldots, \lambda_d) / (\lambda_1 + \cdots + \lambda_d)$. The failure already appears at order $q^2$. Again, this follows from an explicit calculation using the vertex formalism of Section \ref{calc}.
\end{rmk}

As in odd dimensions, restricting the equivariant parameters further, we do get a formula.
\begin{conj}\label{conj for toric 4k}
Let $d \geqslant 4$ such that $d \equiv 0 \ \mathrm{mod}\, 4$ and let $L=\oO_{\mathbb{C}^d}\otimes t_d^{-\ell}$ where $ \ell \in\mathbb{Z}$. Then there exist choices of orientation $o(\lL)$ such that 
\begin{equation}\label{4k eq}Z_{\mathbb{C}^d,L,o({\mathcal{L}})}(q)\Big{|}_{\lambda_1+\cdots+\lambda_{d-1}=\lambda_d=0}=
M_{d-2}(-q)^{\ell}. \end{equation}
\end{conj}
For $d=4$, this was conjectured in \cite{CK} and verified $\mathrm{mod}\, q^7$. In this paper, we provide the following additional evidence:
\begin{thm}\label{main thm on 4k dim}
Conjecture \ref{conj for toric 4k} is true in the following settings:
\begin{itemize}
\item Modulo $q^2$ (Example \ref{verification of one pt for 4k}).
\item When $\ell=1$ (Proposition \ref{thm for equiv sm div}).
\item Modulo $q^7$ when $d=8$ (Proposition \ref{more verify of conj 4k}).
\item Modulo $q^5$ when $d=12$ (Proposition \ref{more verify of conj 4k}).
\end{itemize}
\end{thm}
\begin{exam}\label{verification of one pt for 4k}
(Conjecture \ref{conj for toric 4k} is true modulo $q^2$) Let $d = 4k>0$. According to Example \ref{exam 2}, the coefficient of $q$ equals to
\begin{equation}
(-1)^{o({\mathcal{L}})|_0} e_{T}(L|_0)\cdot
\sqrt{(-1)\cdot\frac{\prod_{i=1}^{2k-1}e_{T}\big(\bigwedge^{2i}T_0\mathbb{C}^d\big)}{\prod_{i=1}^{2k}e_{T}\big(\bigwedge^{2i-1}T_0\mathbb{C}^d\big)}},  \nonumber \end{equation}
where $L=\oO_{\mathbb{C}^d}\otimes t_d^{-\ell}$. Using $\lambda_1+\cdots+\lambda_{d}=0$, this becomes
\begin{equation}(-\ell \lambda_d)\cdot\frac{\prod_{i=1}^{k-1}\prod_{j_1<\cdots<j_{2i}}(\lambda_{j_1}+\cdots+\lambda_{j_{2i}})}
{\prod_{i=1}^{k} \prod_{j_1<\cdots<j_{2i-1}}(\lambda_{j_1}+\cdots+\lambda_{j_{2i-1}})}\cdot
\sqrt{(-1)^{\sigma}\cdot\prod_{j_1<\cdots<j_{2k}}(\lambda_{j_1}+\cdots+\lambda_{j_{2k}})}, \nonumber \end{equation}
where 
$$
\sigma:= 1+ \sum_{i=1}^{k-1} \binom{4k}{2i} + \sum_{i=1}^{k} \binom{4k}{2i-1} = 2^{4k-1} - \frac{1}{2}\binom{4k}{2k} \equiv - \frac{1}{2}\binom{4k}{2k} \ \mathrm{mod}\, 2.
$$
Since $\lambda_d=0$, we have 
\begin{equation}\prod_{j_1<\cdots<j_m}(\lambda_{j_1}+\cdots+\lambda_{j_m})=
\prod_{j_1<\cdots<j_{m-1}<d}(\lambda_{j_1}+\cdots+\lambda_{j_{m-1}})
\cdot\prod_{j_1<\cdots<j_{m}<d}(\lambda_{j_1}+\cdots+\lambda_{j_{m}}).  \nonumber \end{equation}
Using this, we obtain
\begin{align*}
&(-1)^{o({\mathcal{L}})|_0}(-\ell)\cdot\frac{\sqrt{(-1)^{-\frac{1}{2}\binom{4k}{2k}}\cdot\prod_{j_1<\cdots<j_{2k}}(\lambda_{j_1}+\cdots+\lambda_{j_{2k}})}}{\prod_{j_1<\cdots<j_{2k-1}<d}(\lambda_{j_1}+\cdots+\lambda_{j_{2k-1}})} \\
&=(-1)^{o({\mathcal{L}})|_0}(-\ell)\cdot\frac{\sqrt{(-1)^{-\frac{1}{2}\binom{4k}{2k}}\cdot\prod_{j_1<\cdots<j_{2k}<d}(\lambda_{j_1}+\cdots+\lambda_{j_{2k}})
\cdot\prod_{i_1<\cdots<i_{2k-1}<d}(\lambda_{i_1}+\cdots+\lambda_{i_{2k-1}})}}{\prod_{j_1<\cdots<j_{2k-1}<d}(\lambda_{j_1}+\cdots+\lambda_{j_{2k-1}})}  \\
&=(-1)^{o({\mathcal{L}})|_0}(-\ell)\cdot\frac{\sqrt{(-1)^{-\frac{1}{2}\binom{4k}{2k}+\binom{4k-1}{2k}}} \prod_{i_1<\cdots<i_{2k-1}<d}(\lambda_{i_1}+\cdots+\lambda_{i_{2k-1}})}{\prod_{j_1<\cdots<j_{2k-1}<d}(\lambda_{j_1}+\cdots+\lambda_{j_{2k-1}})} \\
&=(-1)^{o({\mathcal{L}})|_0}(-\ell), \end{align*}
where the second equality uses $\lambda_1+\cdots+\lambda_{d-1}=0$ and the third equality uses 
$$
-\frac{1}{2}\binom{4k}{2k}+\binom{4k-1}{2k} =0.
$$
We conclude that for the choice $(-1)^{o({\mathcal{L}})|_0}=1$, Conjecture \ref{conj for toric 4k} holds modulo $q^2$.
\end{exam}

\subsection{Conjecture \ref{conj for toric 4k} for smooth divisors}

\begin{prop}\label{thm for equiv sm div}
Conjecture \ref{conj for toric 4k} is true for $\ell=1$, i.e.~smooth toric divisors.
\end{prop}
\begin{proof}
When $\ell=1$, $L=\oO\otimes (t_d)^{-1}$ corresponds to the smooth divisor $D:=\{x_d=0\}\subset X:=\mathbb{C}^d$. The projection $(t_1, \ldots, t_d) \mapsto (t_1, \ldots, t_{d-1})$ induces an algebraic group isomorphism $T \cong (\mathbb{C}^*)^{d-1}$ allowing us to identify the $T$-action on $D$ with the standard action of $(\mathbb{C}^*)^{d-1}$ on $D$.
For $[Z]\in\Hilb^n(\mathbb{C}^d)^T$ such that $Z\not\subset D$, i.e.~$Z$ is not scheme theoretically contained in $D$, Lemma \ref{vanishing} below implies
\begin{equation}\label{vani} e_{T}(L^{[n]}|_Z)=0. \end{equation}
Hence we only need to calculate 
\begin{equation}\label{equ sum}\sqrt{(-1)^{n}\cdot\frac{e_{T}\big(\Ext_X^{\mathrm{even}}(I_{Z,X},I_{Z,X})\big)}{e_{T}\big(\Ext_X^{\mathrm{odd}}(I_{Z,X},I_{Z,X})\big)}}\cdot e_{T}(L^{[n]}|_Z), \nonumber \end{equation}
for $[Z]\in\Hilb^n(X)^T$ satisfying $Z\subset D\subset X$. Here $I_{Z,X}$ denotes the ideal sheaf of $Z \subset X$ and below we will use $I_{Z,D}$ for the ideal sheaf of $Z$ in $D$. By Proposition \ref{compare ext}, we have
\begin{eqnarray*} \frac{\prod_{i=1}^{2k-1}e_T\big(\Ext^{2i}_X(\iota_*\oO_Z,\iota_*\oO_Z)\big)}{\prod_{i=1}^{2k}e_T\big(\Ext^{2i-1}_X(\iota_*\oO_Z,\iota_*\oO_Z)\big)}=
(-1)^{n} \Bigg(  e_T\big(\Ext^{4k-1}_D(\oO_Z,\oO_Z)\big)^{-1} \prod_{i=1}^{2k-1}\frac{e_T\big(\Ext^{2i}_D(\oO_Z,\oO_Z)\big)}{e_T\big(\Ext^{2i-1}_D(\oO_Z,\oO_Z)\big)} \Bigg)^2 \\
= (-1)^{n}\bigg( e_T\big(H^0(X,\iota_*\oO_Z\otimes \oO_X(D))\big)^{-1} \prod_{i=1}^{2k-1}\frac{e_T\big(\Ext^{2i}_D(\oO_Z,\oO_Z)\big)}{e_T\big(\Ext^{2i-1}_D(\oO_Z,\oO_Z)\big)} \bigg)^2,
\end{eqnarray*}
where $\iota : D \hookrightarrow X$ denotes inclusion and we use 
\begin{equation}\Ext_D^{4k-1}(\oO_Z,\oO_Z)^{*}\cong H^0(D,\oO_Z\otimes K_D)\cong H^0(X,\iota_*\oO_Z\otimes \oO_X(D)).\nonumber \end{equation}
Combining with Proposition \ref{ideal and str sheaf}, we obtain  
\begin{equation}\label{equ sum}\sqrt{(-1)^{n}\cdot\frac{e_{T}\big(\Ext_X^{\mathrm{even}}(I_{Z,X},I_{Z,X})\big)}{e_{T}\big(\Ext_X^{\mathrm{odd}}(I_{Z,X},I_{Z,X})\big)}}\cdot e_{T}(L^{[n]}|_Z)=\pm\frac{e_{T}\big(\Ext_D^{\mathrm{even}}(I_{Z,D},I_{Z,D})\big)}{e_{T}\big(\Ext_D^{\mathrm{odd}}(I_{Z,D},I_{Z,D})\big)}, \nonumber \end{equation}
up to a sign corresponding to the choice of orientation. Note that we are not dividing by zero by Proposition \ref{verify key conj} (because $d-1$ is odd).

On $D$, the specialization $\lambda_1+\cdots+\lambda_{d-1}=\lambda_d=0$ corresponds to restriction of $T \cong \mathbb{C}^{d-1}$ to
\begin{equation}T_0:=\Big\{(t_1,\ldots,t_{d-1},1)\,\big{|}\,t_1\cdots t_{d-1}=1\Big\}, \nonumber \end{equation}
which is isomorphic to the Calabi-Yau subtorus of $D$.
Choosing the right orientations, the result follows from Theorem \ref{thm for toric 2k-1} applied to $D = \mathbb{C}^{d-1}$.
\end{proof}

We are left to prove (\ref{vani}). 
\begin{lem}\label{vanishing}
Let $d>0$ and $L=\oO_{\mathbb{C}^d}(D)$ with $D:=\{x_d=0\}\subset \mathbb{C}^d$. Then there is a $(\mathbb{C}^*)^d$-equivariant isomorphism $L^{[n]} \cong \oO^{[n]} \otimes t_d^{-1}$. Moreover, for any $T$-fixed point $[Z]\in \Hilb^n(\mathbb{C}^d)^T$ such that $Z$ does not scheme-theoretically lie in $D$, we have 
\begin{equation}e_{T}(L^{[n]}|_{Z}) = 0.  \nonumber \end{equation}
\end{lem}
\begin{proof}
Consider the ideal sheaf $\oO(-D) \subset \oO$. This corresponds to the inclusion
\begin{equation}(x_d) \subset \mathbb{C}[x_1,\ldots,x_d].  \nonumber \end{equation}
Hence $\oO(-D) \cong \oO \otimes t_{d}$ and $L \cong \oO \otimes t_{d}^{-1}$. The fibres of $L^{[n]}$ are equal to
\begin{equation}L^{[n]}|_{Z} \cong H^0(L|_Z) \cong H^0(\oO_Z) \otimes t_{d}^{-1},  \nonumber \end{equation}
where all isomorphisms are $(\mathbb{C}^*)^d$-equivariant isomorphisms. 
We obtain a $(\mathbb{C}^*)^d$-equivariant isomorphism
\begin{equation}L^{[n]} \cong \oO^{[n]} \otimes t_d^{-1}.  \nonumber \end{equation}

Suppose $[Z] \in \Hilb^n(\mathbb{C}^d)$ is a $T$-fixed element (also $(\mathbb{C}^*)^d$-fixed by Lemma \ref{T fixed pts are iso}). Then $Z$ corresponds to a $(d-1)$-partition 
$\pi =\{\pi_{i_1\ldots i_{d-1}}\}_{i_1,\ldots, i_{d-1} \geqslant 1}$. If $Z\not\subset D$, i.e.~$Z$ is not scheme-theoretically contained in $D$, then $(x_d) \not\subset I_Z$. Therefore, $\pi_{1\cdots1} > 1$ and the class of $Z$ in $K_{(\mathbb{C}^*)^d}(\bullet)$ contains the term $t_d$. Consequently
\begin{equation*}
e_{(\mathbb{C}^*)^d}(L^{[n]}|_{Z}) = e_{(\mathbb{C}^*)^d}(Z \otimes t_{d}^{-1}) = e_{(\mathbb{C}^*)^d}(1 + \mathrm{other \ terms}) = 
e_{(\mathbb{C}^*)^d}(1)\cdot e_{(\mathbb{C}^*)^d}(\mathrm{other \ terms}) = 0. 
\end{equation*}
The same vanishing holds when setting $\lambda_1 + \cdots + \lambda_d = 0$.
\end{proof}

\section{Vertex formalism and calculations} \label{calc}

\subsection{Vertex formalism}\label{sect vertex formu}
In order to do explicit calculations, and prove Theorem \ref{thm for toric 2k-1}, we develop a vertex formalism in all dimensions closely following the original setup of MNOP \cite{MNOP} (see also \cite[Sect. 3.3]{CK} for the case of toric Calabi-Yau 4-folds).

Let $d \geqslant 2$ and let $Z\subset \mathbb{C}^d$ be a $T$-invariant  zero-dimensional subscheme. Then $Z$ is $(\mathbb{C}^*)^d$-invariant by Lemma \ref{T fixed pts are iso}.
By Section \ref{higher dim partition}, $Z=Z_{\pi}$ for some $(d-1)$-partition $\pi =\{\pi_{i_1\ldots i_{d-1}}\}_{i_1,\ldots, i_{d-1} \geqslant 1}$ and
\begin{equation}\label{corr between part and subscheme}
Z_\pi = \sum_{i_1,\ldots, i_{d-1} \geqslant 1} \sum_{m=1}^{\pi_{i_1\ldots i_{d-1}}} t_1^{i_1-1}\cdots t_{d-1}^{i_{d-1}-1} t_d^{m-1},
\end{equation}
where the sum is over all $i_1,\ldots,i_{d-1} \geqslant1$ for which $\pi_{i_1i_2\ldots i_{d-1}}\geqslant 1$. 

Denote the global section functor by $\Gamma(-)$. Then the local-to-global spectral sequence and \v{C}ech calculation of sheaf cohomology yields
$$
-\dR\mathrm{Hom}(I_Z,I_Z)_0 = \sum_{i} (-1)^i \Big(\Gamma(\mathbb{C}^d, \oO_{\mathbb{C}^d})-\Gamma(\mathbb{C}^d, 
\mathcal{E}{\it{xt}}^i(I_Z,I_Z)) \Big)\in K_{(\mathbb{C}^*)^d}(\bullet),
$$
where we use $H^{>0}(\mathbb{C}^d,-) = 0$. 

Let $R:=\Gamma(\oO_{\mathbb{C}^d}) \cong \mathbb{C}[x_1,\ldots,x_d]$. Consider the class $[I_Z]$ in the equivariant $K$-group 
$K_{(\mathbb{C}^*)^d}(\mathbb{C}^d)$. By identifying $[R]$ with $1$, we obtain a ring isomorphism
$$
K_{(\mathbb{C}^*)^d}(\mathbb{C}^d) \cong \mathbb{Z}[t_1^{\pm},\ldots,t_d^{\pm}].
$$
The Laurent polynomial $\mathsf{P}(I_Z)$ corresponding to $[I_Z]$ under this isomorphism is called the Poincar\'e polynomial of $I_Z$.
For any $w=(w_1,\ldots,w_d) \in \mathbb{Z}^d$, we use multi-index notation
$$
t^w:= t_1^{w_1} \cdots t_d^{w_d}.
$$
Then $[R \otimes t^w] \in K_{(\mathbb{C}^*)^d}(\mathbb{C}^d)$ corresponds to $t^w \in \mathbb{Z}[t_1^{\pm},\ldots,t_d^{\pm}]$.
 
Define an involution $\overline{(\cdot)}$ on $K_{(\mathbb{C}^*)^d}(\mathbb{C}^d)$ by $\mathbb{Z}$-linear extension of
$$
\overline{t^w} := t^{-w}.
$$ 
By definition, the trace map
$$
\tr : K_{(\mathbb{C}^*)^d}(\mathbb{C}^d) \rightarrow \mathbb{Z}(\!(t_1,\ldots,t_d)\!)
$$
corresponds to $(\mathbb{C}^*)^d$-equivariant restriction to the fixed point of $\mathbb{C}^d$. 

Take a $(\mathbb{C}^*)^d$-equivariant graded free resolution
$$
0 \rightarrow F_s \rightarrow \cdots \rightarrow F_0 \rightarrow I \rightarrow 0,
$$
as in \cite{MNOP}, where 
$$
F_i = \bigoplus_j R \otimes t^{d_{ij}},
$$
for certain $d_{ij} \in \mathbb{Z}^d$. Then
\begin{equation} \label{Poin}
\mathsf{P}(I) = \sum_{i,j} (-1)^i t^{d_{ij}}.
\end{equation}
Furthermore, the $(\mathbb{C}^*)^d$-character of $\oO_Z$ is given by \eqref{Zpi}
\begin{equation}\label{Z} 
Z=\sum_{i_1,\ldots, i_{d-1} \geqslant 1} \sum_{m=1}^{\pi_{i_1\ldots i_{d-1}}} t_1^{i_1-1} \cdots t_{d-1}^{i_{d-1}-1} t_d^{m-1}
=\oO_U-I=\frac{1-\mathsf{P}(I)}{(1-t_1)\cdots(1-t_d)}.
\end{equation}
We deduce
\begin{align*}
\dR\mathrm{Hom}(I_Z,I_Z) &= \sum_{i,j,k,l} (-1)^{i+k} \mathrm{Hom}(R \otimes t^{d_{ij}},R \otimes t^{d_{kl}}) \\
&= \sum_{i,j,k,l} (-1)^{i+k} R \otimes t^{d_{kl} - d_{ij}} \\
&= \frac{\mathsf{P}(I)\overline{\mathsf{P}(I)} }{(1-t_1)\ldots(1-t_d)},
\end{align*}
where we used \eqref{Poin} for the last equality. 

Eliminating $\mathsf{P}(I)$ by using \eqref{Z},
$-\dR\mathrm{Hom}(I_Z,I_Z)_0$ is given by
\begin{equation} \label{defV}
\mathsf{V}_\pi := Z_\pi +(-1)^{d}\frac{\overline{Z}_{\pi}}{t_1\cdots t_d} -(-1)^{d}\frac{Z_\pi \overline{Z}_\pi(1-t_1)\cdots(1-t_d)}{t_1\cdots t_d},
\end{equation}
where we re-introduced the subindex $\pi$. Summing up, we have proved the following lemma:
\begin{lem} \label{Vlem}
Let $Z_\pi\subset \mathbb{C}^d$ be a $T$-fixed zero-dimensional subscheme which corresponds to a $(d-1)$-partition $\pi$ via (\ref{corr between part and subscheme}). Then
\begin{equation*} \label{vertexeq}
-\dR\mathrm{Hom}(I_{Z_\pi},I_{Z_\pi})_0=\mathsf{V}_\pi,
\end{equation*}
where the equivariant vertex $\mathsf{V}_\pi$ is defined by \eqref{defV}.
\end{lem}
Now let $d\geqslant 4$ satisfying $d \equiv 0\ \mathrm{mod}\,4$ and make the specialization $t_1\cdots t_d=1$. Then
\begin{align*}
\mathsf{V}_\pi =&\,\sum_{i=1}^{2k}\mathrm{Ext}^{2i-1}(I_{Z_\pi},I_{Z_\pi})
-\sum_{i=1}^{2k-1}\mathrm{Ext}^{2i}(I_{Z_\pi},I_{Z_\pi}) \\
=&\,\sum_{i=1}^{k}\mathrm{Ext}^{2i-1}(I_{Z_\pi},I_{Z_\pi})+\sum_{i=1}^{k}\mathrm{Ext}^{2i-1}(I_{Z_\pi},I_{Z_\pi})^*-
\mathrm{Ext}^{2k}(I_{Z_\pi},I_{Z_\pi}) \\
&\,-\sum_{i=1}^{k-1}\mathrm{Ext}^{2i}(I_{Z_\pi},I_{Z_\pi})-\sum_{i=1}^{k-1}\mathrm{Ext}^{2i}(I_{Z_\pi},I_{Z_\pi})^*,  
\end{align*}
where $\mathrm{Ext}^{2k}(I_{Z_\pi},I_{Z_\pi})$ is self-dual.\,\footnote{Despite being on a non-compact space, $T$-equivariant Serre duality holds by Proposition \ref{ideal and str sheaf}.} Assume Conjecture \ref{key conj} holds for $d$. Then $e_{T}(-\mathsf{V}_\pi)$ is well-defined and given by 
\begin{equation}e_{T}(-\mathsf{V}_\pi)=(-1)^{\sum_{i=1}^{2k-1}(-1)^{i}\mathrm{ext}^{i}(I_{Z_\pi},I_{Z_\pi})}\cdot\Bigg(\frac{\prod_{i=1}^{k-1}e_{T}\big(\mathrm{Ext}^{2i}(I_{Z_\pi},I_{Z_\pi})\big)}
{\prod_{i=1}^{k}e_{T}\big(\mathrm{Ext}^{2i-1}(I_{Z_\pi},I_{Z_\pi})\big)}\Bigg)^2 e_{T}(\mathrm{Ext}^{2k}(I_{Z_\pi},I_{Z_\pi})).
\nonumber \end{equation}
Since the Serre duality pairing on $\mathrm{Ext}^{2k}(I_{Z_\pi},I_{Z_\pi})$ is preserved by $T$, 
there exists a half Euler class $e_{T}\big(\mathrm{Ext}^{2k}(I_{Z_\pi},I_{Z_\pi}),Q\big)$ as in (\ref{half euler class equivariant}) satisfying 
\begin{equation}e_{T}\big(\mathrm{Ext}^{2k}(I_{Z_\pi},I_{Z_\pi}),Q\big)^2=(-1)^{\frac{1}{2}
\mathrm{ext}^{2k}(I_{Z_\pi},I_{Z_\pi})}\cdot e_{T}\big(\mathrm{Ext}^{2k}(I_{Z_\pi},I_{Z_\pi})\big). 
\nonumber \end{equation}
Again using Proposition \ref{ideal and str sheaf} and Serre duality on a compactification of $\mathbb{C}^d$, a Hirzebruch-Riemann-Roch calculation shows
$$
\sum_{i=1}^{2k-1}(-1)^i\mathrm{ext}^{i}(I_{Z_\pi},I_{Z_\pi}) = -n-\frac{\mathrm{ext}^{2k}(I_{Z_\pi},I_{Z_\pi}) }{2},
$$
where $n = |\pi|$ denotes the length of $Z_\pi$. We conclude
\begin{equation*}\label{equi square equality}e_{T}(-\mathsf{V}_\pi)=(-1)^{|\pi|}\cdot
\Bigg(e_{T}(\mathrm{Ext}^{2k}\big(I_{Z_\pi},I_{Z_\pi}),Q\big) \cdot \frac{\prod_{i=1}^{k-1}e_{T}\big(\mathrm{Ext}^{2i}(I_{Z_\pi},I_{Z_\pi})\big)}
{\prod_{i=1}^{k}e_{T}\big(\mathrm{Ext}^{2i-1}(I_{Z_\pi},I_{Z_\pi})\big)}\Bigg)^2. \end{equation*}
\begin{defi} \label{wpi sq root}
Assume Conjecture \ref{key conj} holds for $d$ (where $d\geqslant 4$, $d \equiv 0\ \mathrm{mod}\, 4$). Let $\pi$ be a $(d-1)$-partition and let
$\mathsf{V}_\pi$ be the expression defined by \eqref{defV}. We define
$$
\mathsf{w}_\pi:=\sqrt{(-1)^{|\pi|}\cdot e_{T}(-\mathsf{V}_\pi)}
\in \mathbb{Q}(\lambda_1,\ldots,\lambda_d)/(\lambda_1+\cdots+\lambda_d),
$$
which is only defined up to a sign, i.e.~$\sqrt{\alpha^2}$ stands for $\pm \alpha$. 
\end{defi}

From Lemma \ref{Vlem} and Definition \ref{wpi sq root}, we conclude the following.
\begin{prop}
Assume Conjecture \ref{key conj} holds for $d$ (where $d\geqslant 4$, $d \equiv 0\ \mathrm{mod}\, 4$). Let $\pi$ be a $(d-1)$-partition and 
$Z_\pi$ be the corresponding zero-dimensional subscheme. Then we have 
$$
\sqrt{(-1)^{|\pi|}\cdot\frac{e_{T}\big(\Ext^{\mathrm{even}}(I_{Z_\pi},I_{Z_\pi})\big)}{e_{T}\big(\Ext^{\mathrm{odd}}(I_{Z_\pi},I_{Z_\pi})\big)}}=\mathsf{w}_{\pi},
$$
where left-hand side and right-hand side are only defined up to sign.
\end{prop}

\subsection*{Tautological insertion}
Let $L$ be $(\mathbb{C}^*)^d$-equivariant line bundle on $\mathbb{C}^d$. Then
$$
L=\oO_{\mathbb{C}^d} \otimes t_1^{u_1}\cdots t_d^{u_d}
$$
for some $(u_1,\ldots,u_d) \in \mathbb{Z}^d$. Let $\pi$ be a $(d-1)$-partition and $Z_\pi$ the corresponding zero-dimensional subscheme. We define 
$$
L_\pi(u_1,\ldots,u_d) := e_{T}\big( H^0(\mathbb{C}^d,\oO_{Z_\pi}\otimes L)\big) \in \mathbb{Q}(\lambda_1,\ldots,\lambda_d) / (\lambda_1+\cdots+\lambda_d),
$$
where 
\begin{equation} 
H^0(\mathbb{C}^d,\oO_{Z_\pi}\otimes L)=\sum_{i_1,\ldots, i_{d-1}\geqslant 1}
\sum_{m=1}^{\pi_{i_1\ldots i_{d-1}}} t_1^{u_1+i_1-1} \cdots t_{d-1}^{u_{d-1}+i_{d-1}-1} t_d^{u_d+m-1}.
\nonumber \end{equation}
By definition of the $T$-equivariant structure on $L^{[n]}$, we have
$$
e_T(L^{[n]})|_Z = L_{\pi}(u_1,\ldots,u_d). 
$$
In summary, with the notation introduced above, we may rewrite the $T$-equivariant counts in Definition \ref{def of equi invs for 4k} as follows
\begin{equation}Z_{\mathbb{C}^d,L,o({\mathcal{L}})}(q)=\sum_{(d-1)\textrm{-partition} \ \pi} L_{\pi}(u_1,\ldots,u_d) \, \mathsf{w}_{\pi} \, q^{|\pi|}, \nonumber \end{equation}
where the choices of orientation $o({\mathcal{L}})$ precisely correspond to the choices of sign for $\mathsf{w}_\pi$.

\subsection{Proof of Theorem \ref{thm for toric 2k-1}}\label{pf of odd conj}

The following proof is almost identical to the original argument for $d=3$ in \cite[Sect.~4.10, 4.11]{MNOP}. In order to convince the reader that the argument works in any odd dimension $d\geqslant 3$, we include the details. 
\begin{proof}
Let $d = 2k+1\geqslant 3$ and recall that we work with respect to the Calabi-Yau torus $T = \{t_1 \cdots t_{d} = 1\} \subset (\C^*)^d$. Let $\pi$ be a $(d-1)$-partition with corresponding zero-dimensional subscheme $Z_\pi$. We claim
$$
\frac{e_T\big(\Ext^{\mathrm{even}}(I_Z,I_Z)\big)}{e_T\big(\Ext^{\mathrm{odd}}(I_Z,I_Z)\big)} = (-1)^{|\pi|}.
$$
Define\footnote{Note that in \cite{MNOP}, a more complicated definition of $\mathsf{V}_\pi^+$ is needed in the presence of ``infinite legs''.}
\begin{align*}
\mathsf{V}_\pi^+ &:= Z_\pi - Z_\pi \overline{Z}_\pi \frac{(1-t_1) \cdots (1-t_{2k})}{t_1 \cdots t_{2k}}, \\
\mathsf{V}_\pi^- &:= \mathsf{V}_\pi - \mathsf{V}^+_\pi.
\end{align*}
Using $t_1 \cdots t_d=1$, a short calculation  shows that this splitting has the following key property
\begin{equation} \label{Vplusbar}
\overline{\mathsf{V}}_\pi^+ = - \mathsf{V}_\pi^-.
\end{equation}
Although $\mathsf{V}_\pi$ does not have a $T$-fixed term (Proposition \ref{verify key conj}), the new object $\mathsf{V}_\pi^+$ could have a $T$-fixed term. Suppose the $T$-fixed term of $\mathsf{V}_\pi^+$ is \emph{even}. Then \eqref{Vplusbar} implies
$$
\frac{e_T\big(\Ext^{\mathrm{even}}(I_Z,I_Z)\big)}{e_T\big(\Ext^{\mathrm{odd}}(I_Z,I_Z)\big)} = (-1)^{\mathsf{V}_\pi^+(1,\ldots,1)} = (-1)^{Z_\pi(1,\ldots,1)} = (-1)^{|\pi|}.
$$

We are left to show that the constant term of $\mathsf{V}_\pi^+$ is even. The proof is inductive on $|\pi|$. Suppose $b \in \pi$ is an extremal box on the highest level in the $x_{d}$-direction and suppose this box is located at $b = (b_1, \ldots, b_d)$ (where we use the coordinate of the corner closest to the origin).\,\footnote{``Extremal'' means that all boxes $b'=(b_1', \ldots, b_d')$ adjacent to $b$ satisfy $b_i' \leq b_i$ for all $i$.} We show that the contribution of $b$ to the constant term of $\mathsf{V}_\pi^+$ is even. This contribution comes from the following terms:
\begin{enumerate}
\item[(i)] $\Bigg( t_1^{b_1} \cdots t_d^{b_d} \Bigg)^T$,
\item[(ii)] $\Bigg( b \overline{b'} \frac{(1-t_1) \cdots (1-t_{2k})}{t_1 \cdots t_{2k}} \Bigg)^T$,
\item[(iii)] $\Bigg( b' \overline{b} \frac{(1-t_1) \cdots (1-t_{2k})}{t_1 \cdots t_{2k}} \Bigg)^T$,
\end{enumerate}
where $(\cdot)^T$ denotes $T$-fixed part, $b' \in \pi$ runs over all boxes, and where we only include contribution (ii) when $b'=b$. Observation: suppose $b' = (b_1', \ldots, b_{d-1}',z) \in \pi$ and $b^{\prime \prime} = (b_1', \ldots, b_{d-1}',z-1) \in \pi$, then the contribution of $b'$ to (ii) \emph{equals} the contribution of $b^{\prime\prime}$ to (iii), so they cancel modulo 2. We are therefore only left to consider the following contributions:
\begin{enumerate}
\item[(a)] Contribution of $b' = (b_1', \ldots, b_{d-1}',z) \in \pi$ to (iii), where $(b_1', \ldots, b_{d-1}',z+1) \not\in \pi$.
\item[(b)] Contribution of $b' = (b_1', \ldots, b_{d-1}',0) \in \pi$ to (ii).
\item[(c)] Contribution to (i).
\end{enumerate}
Contributions of type (a) are always zero. Indeed, if $z = b_d$, then $b' \neq b$ (recall that $b$ is extremal and if $b=b'$, we do not count (iii)). If $z<b_d$, then $b_{i}' >b_i$ for some $i=1, \ldots, d-1$. Neither case leads to $T$-fixed contributions.

Next, we analyse contributions of type (b). Suppose $b\neq(n,\ldots,n)$, i.e.~$b$ does not lie on the diagonal. In this case, the contribution to (b) is of the following form $0,\, \pm 2,\, \pm 4, \ldots,\, \pm2^{d-1}$, which all vanish modulo 2. Since we do not have contributions of type (c), we are done. Suppose $b=(n, \ldots, n)$, i.e.~$b$ lies on the diagonal. Then only $b' = (0, \ldots, 0)$ contributes to (b) and its contribution is 1. But $b$ also contributes 1 to (c), so these contributions cancel. We are done.
\end{proof}

\subsection{Computer calculations}

Using the vertex formalism developed in Section \ref{sect vertex formu}, we wrote a computer program in Maple which verifies Conjecture \ref{conj for toric 4k} in the following cases. The case $d=4$ has been verified modulo $q^7$ in \cite{CK}.
\begin{prop}\label{more verify of conj 4k}  
$ $
\begin{enumerate}
\item[(1)] Conjecture \ref{conj for toric 4k} holds for $\mathbb{C}^8$ modulo $q^7$.  
\item[(2)] Conjecture \ref{conj for toric 4k} holds for $\mathbb{C}^{12}$ modulo $q^5$. 
\end{enumerate}
\end{prop}

In fact, experimentation with our Maple program leads us to a further conjecture.
\begin{conj}\label{sign unique}
The choices of orientation for which equation \eqref{4k eq} holds are unique.
\end{conj}

Indeed, we verifiy this in the following cases.
\begin{prop}\label{verify sign unique}  
$ $
\begin{enumerate}
\item[(1)] Conjecture \ref{sign unique} holds for $\mathbb{C}^8$ modulo $q^7$.  
\item[(2)] Conjecture \ref{sign unique} holds for $\mathbb{C}^{12}$ modulo $q^5$. 
\end{enumerate}
\end{prop}

\begin{rmk}
A priori the number of choices of orientation for the coefficient of $q^6$ of $Z_{\mathbb{C}^8,L,o({\mathcal{L}})}(q)$ equals 
$$
2^{2024} \approx 2 \cdot 10^{609},
$$
which is enormous. Nevertheless, we are able to go through all possible orientations due to the following observation. After specializing $\lambda_1+ \cdots+\lambda_7 = \lambda_8=0$ and for any choice of orientations, the coefficient of $q^6$ of $Z_{\mathbb{C}^8,L,o({\mathcal{L}})}(q)$ is a degree 6 polynomial in $\ell$. Starting with $i=0$, we inductively consider the coefficient $a_{6-i}$ of $\ell^{6-i}$ and we observe:
\begin{itemize}
\item If all 7-partitions $\pi$ contributing to $a_{6-i}$ are chosen with orientation such that $\omega_\pi > 0$, then we obtain the coefficient $b_{6-i}>0$ of $\ell^{6-i}$ coming from expanding $M_6(-q)^\ell$. 
\item Since $b_{6-i}>0$, any other choice of orientation would lead to $a_{6-i} < b_{6-i}$.
 \end{itemize} 
 Another simplification occurs by noting that the weight 
 $$
 \pm L_{\pi}(0,\ldots, 0,-\ell) \, \mathsf{w}_{\pi} \Big|_{\lambda_1+ \cdots+\lambda_7 = \lambda_8=0}
 $$ 
assigned to a 7-partition $\pi$ is invariant under permuting the coordinate axes $x_1, \ldots, x_7$ (but not $x_8$). Therefore, we only have to work with 7-partitions up to these permutation symmetries. A similar strategy works in all cases where we checked uniqueness.
\end{rmk}

\subsection{Application to enumerating partitions}

In this section we discuss relations of Conjectures \ref{conj for toric 4k} and \ref{sign unique} to the enumeration of $(d-1)$-partitions, where $d \geqslant4$ and $d \equiv 0\ \mathrm{mod}\, 4$. As part of Conjecture \ref{conj for toric 4k}, we state that the specialization $\lambda_1+\cdots+\lambda_{d-1}=\lambda_d=0$ is well-defined. Using the topological vertex discussed earlier in this section, and implemented into a Maple routine for $d=4,\, 8,\, 12$, we conjecture the following.
\begin{conj} \label{specconj}
Let $d\geqslant4$ such that $d \equiv 0\ \mathrm{mod}\, 4$. Let $\pi=\{\pi_{i_1\ldots i_{d-1}}\}_{i_1,\ldots, i_{d-1} \geqslant 1}$ be a $(d-1)$-partition. Consider $\mathsf{w}_\pi$ with the unique sign given by Conjecture \ref{sign unique}. Then 
$$
L_\pi(0,0,\ldots,-\ell) \, \mathsf{w}_\pi  \big|_{\lambda_1+\lambda_2+\cdots+\lambda_{d-1}=\lambda_{d}=0} = (-1)^{|\pi|}\,\omega_\pi \cdot \prod_{m=1}^{\pi_{1\cdots 1}} (\ell-(m-1)),
$$
for some $\omega_\pi \in \mathbb{Q}_{>0}$.
\end{conj}
Geometrically, this specialization corresponds to taking $X = \mathbb{C}^{d}$ and $D = \{x^{\ell}_d=0\} \subset \mathbb{C}^{d}$. Then $L = \oO(D) \cong \oO \otimes t_{d}^{-\ell}$ and the canonical section of $L^{[n]}$ on $\Hilb^n(\mathbb{C}^{d})$ cuts out the sublocus of zero-dimensional subschemes $Z$ contained in $D$ (Proposition \ref{section s}). The $T$-fixed points of this locus correspond precisely to the $(d-1)$-partitions $\pi$ of height $\pi_{1\cdots1}\leqslant \ell$. 

We provide the following evidence for the above conjecture.
\begin{prop}\label{weight wpi check}
Conjecture \ref{specconj} is true in the following settings:
\begin{itemize}  
\item When $\ell=1$ (in which case $\omega_\pi$ can be taken to be 1).
\item When $d=4$ and $|\pi|\leqslant 6$ (see \cite[Prop.~4.2]{CK}).  
\item When $d=8$ and $|\pi|\leqslant 6$.  
\item When $d=12$ and $|\pi|\leqslant 4$. 
\end{itemize}
Furthermore, the absolute value of Conjecture \ref{specconj} is true in the following cases:
\begin{itemize}
\item For the list of individual $3$-partitions of sizes 7--15 in \cite[App.~A]{CK}.
\item For the list of individual $7$-partitions of sizes 9, 10, 14 in Remark \ref{indivpart}.
\end{itemize}
\end{prop}
\begin{proof}
The case $\ell=1$ follows from the proof of Proposition \ref{thm for equiv sm div}. The other cases follow from our implementation into Maple of the topological vertex for $\mathbb{C}^4$, $ \mathbb{C}^8$, $\mathbb{C}^{12}$ and computer calculations.
\end{proof}
We have the following application of Conjectures \ref{conj for toric 4k}, \ref{sign unique}, and \ref{specconj}.
\begin{thm}\label{4k-1 partition counting}
Assume Conjectures \ref{conj for toric 4k}, \ref{sign unique}, and \ref{specconj} are true for $d$ (where $d \geqslant 4$ satisfies $d \equiv 0\  \mathrm{mod}\, 4$). Then
\begin{equation*}\label{solid part count}
\sum_{(d-1)\textrm{-}\mathrm{partitions} \, \pi} \omega_\pi \,t^{\pi_{1\cdots1}} \, q^{|\pi|} = e^{t (M_{d-2}(q)-1)}, \end{equation*}
where $t$ is a formal variable. In particular, by setting $t = 1$, we obtain
\begin{equation}
\sum_{(d-1)\textrm{-}\mathrm{partitions} \, \pi} \omega_\pi \, q^{|\pi|} = e^{M_{d-2}(q)-1}. \nonumber 
\end{equation}
\end{thm}
\begin{proof}
This is a simple calculation carried out for $d=4$ in \cite[Thm. 2.19]{CK}.
\end{proof}

We end this paper by assigning a combinatorially defined weight $\omega_\pi^c \in \mathbb{Q}_{>0}$ to any $n$-partition $\pi$ for any $n \geqslant 1$. We prove an analogue of Theorem \ref{4k-1 partition counting} with $\omega_\pi$ replaced by $\omega_\pi^c$ (and without assuming any conjecture). When $n = 4k-1 \geqslant 3$, we conjecture $\omega_\pi = \omega_\pi^c$ and provide evidence in many cases.

The following notions were introduced for $n = 2$ in \cite[Def.~4.5]{CK}.
\begin{defi} \label{binary rep}
Let $n \geqslant 1$. Let $\xi=\{\xi_{i_1\ldots i_n}\}_{i_1,\ldots, i_n \geqslant 1}$ be an $n$-partition. We define its binary representation as the sequence of integers 
$\{\xi(i_1,\ldots,i_{n+1})\}_{i_1\ldots,i_{n+1} \geqslant 1}$ given by
$$
\xi(i_1,\ldots,i_n,i_{n+1})= \left\{\begin{array}{cc} 1 & \,\textrm{ if\ } i_{n+1} \leqslant \xi_{i_1\ldots i_n} \\ 0 & \textrm{otherwise} \end{array}\right..
$$
\end{defi}
\begin{defi}\label{combinatoric wpi}
Let $n \geqslant 1$ and $\pi=\{\pi_{i_1\ldots i_{n+1}}\}_{i_1,\ldots, i_{n+1} \geqslant 1}$ be an $(n+1)$-partition. Consider all possible sequences of integers $\{m_\xi\}_\xi$, where the index $\xi$ runs over all (non-empty) $n$-partitions and $m_\xi \in \mathbb{Z}_{\geqslant 0}$. Define the following collection
$$
\mathcal{C}_\pi:=\Bigg\{ \{m_\xi\}_\xi \ \Bigg| \ \pi_{i_1\ldots i_{n+1}}= \sum_\xi m_\xi \cdot \xi(i_1,\ldots, i_{n+1}) \ \textrm{for all} \ i_1,\ldots,i_{n+1}\geqslant1 \Bigg\}.
$$
Using $\mathcal{C}_\pi$, we define
$$
\omega^c_\pi:=\sum_{\{m_\xi\}_\xi \in \mathcal{C}_\pi} \prod_\xi \frac{1}{(m_\xi)!}.
$$
\end{defi}

\begin{rmk}
For each $\{m_\xi\}_\xi\in\mathcal{C}_\pi$, we have 
\begin{equation}|\pi|=\sum_{\xi}m_{\xi}\cdot |\xi|. \nonumber \end{equation}
Hence, $m_{\xi}=0$ if $|\xi|$ is large. So the collection $\mathcal{C}_{\pi}$ is a finite set, and for each $\{m_\xi\}_\xi\in\mathcal{C}_\pi$, 
there are only finitely many nonzero $m_{\xi}$.
\end{rmk}
The combinatorial weight defined above gives the following generating series.
\begin{prop}\label{comb proof}
For all $n \geqslant 1$, we have
\begin{equation}
\sum_{n\textrm{-}\mathrm{partitions}\, \pi} \omega^c_\pi\,t^{\pi_{1\cdots1}} q^{|\pi|} = e^{t(M_{n-1}(q)-1)}, \nonumber 
\end{equation}
where $M_0(q) := \frac{1}{1-q}$.
\end{prop}
\begin{proof}
The proof of \cite[Prop. 2.30]{CK}, where $n=3$, generalizes immediately.
\end{proof}

\begin{rmk}
In fact, Definitions \ref{binary rep}, \ref{combinatoric wpi}, and Proposition \ref{comb proof} for $n=3$ in \cite{CK} initiated the present project. We realized that these results all hold for general $n$ and our goal was to give an interpretation in terms of zero-dimensional counts on Hilbert schemes of $\C^{n+1}$ when $n+1 \geqslant 4$ and $n+1 \equiv 0\ \mathrm{mod}\, 4$. See \cite[Sect.~1.4, Remark~4.12]{CK}.
\end{rmk}

We end with the following conjecture (see \cite{CK} for the case $d=4$).
\begin{conj}\label{compare wpi}
Let $d \geqslant 4$ such that $d \equiv 0\ \mathrm{mod}\,4$. Then $\omega_\pi=\omega^c_\pi$ for any $(d-1)$-partition $\pi$.
\end{conj}
\begin{prop}\label{prop compare wpi}
Conjecture \ref{compare wpi} is true in the following cases:
\begin{itemize}
\item For any $(d-1)$-partition $\pi$ satisfying $\pi_{1 \cdots 1} = 1$.
\item For $3$-partitions $\pi$ of size $|\pi| \leqslant 6$ (see \cite[Prop.~4.14]{CK}).
\item For $7$-partitions $\pi$ of size $|\pi| \leqslant 6$.
\item For $11$-partitions $\pi$ of size $|\pi| \leqslant 4$.
\end{itemize}
In \cite[App.~A]{CK}, we verified Conjecture \ref{compare wpi}, up to sign, for a certain list 3-partitions of size 7--15. Finally, we verified Conjecture \ref{compare wpi}, up to sign, for the $7$-partitions of size $9,10,14$ in Remark \ref{indivpart} below.
\end{prop}
\begin{proof}
For any partition satisfying $\pi_{1 \cdots 1}=1$, we have $\omega_\pi^c = 1$ and $\omega_\pi = 1$. The latter follows from the proof of Proposition \ref{thm for equiv sm div}. The other cases follow from our implementation into Maple of the topological vertex for $\mathbb{C}^4$, $ \mathbb{C}^8$, $\mathbb{C}^{12}$ and computer calculations.
\end{proof}

\begin{rmk} \label{indivpart}
For the following 7-partitions, we check that $|\omega_\pi|$ and $\omega_\pi^c$ agree. These checks are only up to sign, because we have not verified Conjecture \ref{sign unique} for $|\pi| > 6$.
\begin{itemize}
\item $Z_\pi = 1+t_1+t_2+t_3+t_4+t_5+t_6+t_7+t_8$ then $$|\omega_\pi| = \omega_\pi^c = 64.$$ Here $|\omega_\pi| = 64$ follows from our Maple implementation of the topological vertex for $\mathbb{C}^8$. The equality $\omega_\pi^c=64$ follows from Definition \ref{combinatoric wpi} and the following calculation
$$
\omega_\pi^c = 1 + 7 + \binom{7}{2} + \binom{7}{3} = 64.
$$
Here $1$ corresponds to the ``superposition'' of $\xi_1=1$ and $\xi_2=1+t_1+ \cdots + t_7$; $7$ corresponds to the superposition of $\xi_1=1+t_{i_1}$ and $\xi_2=1+t_{i_2}+t_{i_3}+ t_{i_4} + t_{i_5} + t_{i_6} + t_{i_7}$ where $i_1, \ldots, i_7 \in \{1, \ldots, 7\}$ are mutually distinct; $\binom{7}{2}$ corresponds to the ``superposition'' of $\xi_1=1+t_{i_1}+t_{i_2}$ and $\xi_2=1+t_{i_3}+t_{i_4}+ t_{i_5} +t_{i_6}+t_{i_7}$ where $i_1, \ldots, i_7 \in \{1, \ldots, 7\}$ are mutually distinct; and $\binom{7}{3}$ corresponds to the ``superposition'' of $\xi_1=1+t_{i_1}+t_{i_2}+t_{i_3}$ and $\xi_2=1+t_{i_4}+t_{i_5}+ t_{i_6}+t_{i_7} $ where $i_1, \ldots, i_7 \in \{1, \ldots, 7\}$ are mutually distinct.
\item $Z_\pi =1+t_1+t_2+t_3+t_4+t_5+t_6+t_7+t_8+t_8^2$ then $$|\omega_\pi| = \frac{729}{2}= \frac{1}{2}+7+\binom{7}{2}+\binom{7}{3}+\binom{7}{2}+7 \cdot \binom{6}{2} + \frac{1}{2} \cdot 7 \cdot \binom{6}{3} + \frac{1}{2} \cdot \binom{7}{2} \cdot \binom{5}{2} = \omega_\pi^c.$$
\item $Z_\pi =1+t_1+t_2+t_3+t_1t_2+t_1t_3+t_2t_3+t_1t_2t_3+t_4+t_5+t_6+t_7+t_8+t_8^2$ then $$|\omega_\pi| = \frac{81}{2} = \frac{1}{2} + 4 + \binom{4}{2} + \binom{4}{2} + 4 + 4 \cdot 3 + 1 + 4 + \frac{1}{2} \cdot \binom{4}{2} = \omega_\pi^c.$$
\end{itemize}
\end{rmk}

\end{document}